\documentclass[11pt]{amsart} \usepackage{latexsym, amssymb, stmaryrd, times, verbatim, amsmath}
\usepackage[T1]{fontenc}
\usepackage[multiple]{footmisc}

\newtheorem{thm}{Theorem}[section]
\newtheorem{lemma}[thm]{Lemma}
\newtheorem{prop}[thm]{Proposition}
\newtheorem{cor}[thm]{Corollary}
\newtheorem{question}[thm]{Question}
\newtheorem{fact}[thm]{Fact}
\newtheorem{conj}[thm]{Conjecture}

\theoremstyle{definition}
\newtheorem{df}[thm]{Definition}

\newtheorem{example}[thm]{Example}

\newtheorem{rmk}[thm]{Remark}

\theoremstyle{remark}

\makeatletter

\def\dotminussym#1#2{%
  \setbox0=\hbox{$\m@th#1-$}%
  \kern.5\wd0%
  \hbox to 0pt{\hss\hbox{$\m@th#1-$}\hss}%
  \raise.6\ht0\hbox to 0pt{\hss$\m@th#1.$\hss}%
  \kern.5\wd0}
\newcommand{\dotminus}{\mathbin{\mathpalette\dotminussym{}}}

\renewcommand{\r}{\mathbb{R}}
\newcommand{\Z}{\mathbb{Z}}

\newcommand{\curly}[1]{\mathcal{#1}}

\newcommand{\B}{\curly{B}}

\newcommand{\n}{\mathbb{N}}

\renewcommand{\to}{\rightarrow}

\def \Diag{\operatorname{Diag}}

\def \<{\langle}
\def \>{\rangle}

\def \*Z {{{^*}\Z}}
\def \((  {(\!(}
\def \)) {)\!)}

\def \O{\operatorname{O}}

\def \tp{\operatorname{tp}}

\numberwithin{equation}{section}

\def \Th{\operatorname{Th}}
\def \R{\mathcal R}
\def \u{\mathcal U}

\def \O{\mathcal O}
\def \nuc{\operatorname{nuc}}

\def \P{\mathbb{P}}
\def \etp{\operatorname{etp}}

\newcommand{\cstar}{$\mathrm{C}^*$}

\allowdisplaybreaks[2]

\def \Q{\mathcal{Q}}

\def \LLP{\operatorname{LLP}}
\def \NG{\mathbb{NG}}
\def \GS{\mathbb{GS}}
\def \appdiag{\operatorname{AppDiag}}
\title{Enforceable operator algebras}
\author{Isaac Goldbring}
\thanks{I. Goldbring was partially supported by NSF CAREER grant DMS-1708802.}

\address{Department of Mathematics\\
University of California, Irvine\\
Irvine, CA 92697\\
 USA}
\email{isaac@math.uci.edu}

\begin{document}

\begin{abstract}
We adapt the classical notion of building models by games to the setting of continuous model theory.  As an application, we study to what extent canonical operator algebras are enforceable models.  For example, we show that the hyperfinite II$_1$ factor is an enforceable II$_1$ factor if and only if the Connes Embedding Problem has a positive solution.  We also show that the set of continuous functions on the pseudoarc is an enforceable model of the theory of unital, projectionless, abelian \cstar-algebras and use this to show that it is the prime model of its theory.
\end{abstract}
\maketitle

\tableofcontents

%
%
%

\section{Introduction}

The technique of model-theoretic forcing and, more specifically, the approach via games, is a well-developed part of classical model theory and has found applications in algebraic areas such as in the model theory of groups.  (Throughout this article, our main reference for this topic is the wonderful book \cite{hodges}.)  While model-theoretic forcing has been transported to the setting of continuous logic (see \cite{BI}, \cite{munster}, \cite{farahmagidor}) and has found nice applications to functional analysis and operator algebras, the approach via games has yet to make its continuous appearance.  In this paper, we present the approach of building models by games in the continuous setting and use it to prove some new results in the model theory of operator algebras.  In addition to these aforementioned applications, we believe that the approach to model-theoretic forcing via games is much easier to understand for the non-logician than the other presentations in the literature.  Moreover, the game approach allows one to consider an important notion not readily apparent in the other approaches, namely that of an \emph{enforceable structure}. 

Let us briefly describe the game here.  To be concrete, let us choose a particular setting, say the setting of tracial von Neumann algebras.  Let us fix a set $C$ of distinct symbols that are to represent generators of a tracial von Neumann algebra that two players (traditionally named $\forall$ and $\exists$) are going to build together (albeit adversarially).  The two players take turns playing finite sets of expressions of the form $|\|p(c)\|_2-r|<\epsilon$, where $c$ is a tuple of variables, $p(c)$ is a $*$-polynomial, and each player's move is required to extend the previous player's move.  These sets are called (open) \emph{conditions}.  Moreover, these conditions are required to be \emph{satisfiable}, meaning that there should be some tracial von Neumann algebra $M$ and some tuple $a$ from $M$ such that $|\|p(a)\|_2-r|<\epsilon$ for each such expression in the condition.

We play this game for $\omega$ many steps.\footnote{Perhaps to the disappointment of the operator algebra audience, in this article, $\omega$ denotes the first infinite ordinal, not an ultrafilter.}  At the end of this game, we have enumerated some countable, satisfiable set of expressions.  Provided that the players behave, they can ensure that the final set of expressions yields complete information about all $*$-polynomials over the variables $C$ (that is, for each $*$-polynomial $p(c)$, there should be a unique $r$ such that the play of the game implies that $\|p(c)\|_2=r$) and that this data describes a countable, dense $*$-subalgebra of a unique tracial von Neumann algebra, which is often called the \emph{compiled structure}.   

The question then becomes:  what kinds of properties can we force the compiled structure to have?  More precisely, given a property $P$, is there a strategy for $\exists$ so that, regardless of player $\forall$'s moves, if $\exists$ follows the strategy, then the compiled structure will have that property?  If this is the case, we call the property $P$ an \emph{enforceable} property of tracial von Neumann algebras.  It is natural to ask:  are there any interesting enforceable properties of tracial von Neumann algebras?  We will later see that it is enforceable that the compiled structure is a McDuff II$_1$ factor.  (Recall that a II$_1$ factor is McDuff if it tensorially absorbs the hyperfinite II$_1$ factor $\R$.)

Of central importance in this paper is a seemingly extraordinary case:  Suppose that the property $P$ is the property of being isomorphic to a particular separable II$_1$ factor $\mathcal{E}$.  If this property is enforceable, we say that $\mathcal{E}$ is the \emph{enforceable} II$_1$ factor.  Clearly, there can be at most one enforceable II$_1$ factor.  But is there one?  While this may seem like an extreme possibility that never happens, there are many situations in classical logic where there is an enforceable structure.  For example, if one plays the discrete version of the above game with fields of a fixed characteristic, then the algebraic closure of the prime field is the enforceable structure. 

Again, we ask:  is there an enforceable II$_1$ factor?  The answer is connected to arguably the most famous open problem in the theory of II$_1$ factors, namely the \emph{Connes Embedding Problem}.  Recall that the Connes Embedding Problem asks whether or not every II$_1$ factor embeds into an ultrapower of $\R$.  One of the main results of the current paper is that the Connes Embedding Problem has a positive solution if and only if $\R$ is the enforceable II$_1$ factor.  We will also show that if one restricts the above game to only playing conditions that are satisfiable in tracial von Neumann algebras that do embed into an ultrapower of $\R$, then $\R$ is indeed the enforceable structure for this game.  We also prove analogous results for various games concerned with \cstar-algebras and operator spaces and systems.

The original motivation for this work was model-theoretic questions around the \emph{pseudoarc} $\P$.  Using the game-theoretic machinery, we will prove that $C(\P)$ is the prime model of its theory, a result which had yet to be proven thus far.

Let us conclude by outlining the contents of this paper.  In Section 2, we carefully describe the aforementioned game in the setting of an arbitrary continuous language and describe how the associated notion of forcing connects with the presentations of forcing that have already appeared in the literature.  In Section 3, we describe the important notion of a \emph{finite-generic} structure.  These are structures for which forcing and truth coincide.  Many of our applications rely on foundational properties of finite-generic structures and so a careful presentation of these results is needed.  In Section 4, we describe the aforementioned application to the model theory of the pseudoarc.  In Section 5, we discuss the already described connection between enforceable models and embedding problems in operator algebras.  In the final section, we prove the so-called dichotomy theorem, which shows that, for certain kinds of theories (including many of those appearing in operator algebras), either there is an enforceable structure or else, for any enforceable property $P$, there are continuum many nonisomorphic separable structures with property $P$.  We speculate on how this theorem might provide a new approach to CEP and other embedding problems in operator algebras.

We would like to thank Ilijas Farah and Bradd Hart for useful conversations in connection with this work.

\subsection{Preliminaries, notations, and conventions}
We will assume that the reader is familiar with the basics of continuous logic.  Standard references are \cite{bbhu} and \cite{munster}, the latter of which stresses applications to operator algebras.  In this subsection, we will just collect a few preliminary notions that are of central importance in this paper and deserve to be recalled.  We also take the opportunity to set up some notation.

Fix a continuous signature $L$, which, for simplicity, we assume is $1$-sorted and of diameter bounded by $1$.  For any $n\geq 1$, there is a natural seminorm on the space of all $L$-formulae with free variables amongst the variables $x=x_1,\ldots,x_n$, namely
$$\|\varphi(x)\|:=\sup\{|\varphi(a)| \ : \ A \text{ is an $L$-structure and }a\in A^n\}.$$

By a \emph{restricted} $L$-formula, we mean an $L$-formula constructed using only the unary connectives $\frac{1}{2}$ and $\neg$ (here, $\neg r:=1- r$) and the binary connective $\dotplus$ (truncated addition).  The family of restricted $L$-formulae is dense (with respect to the seminorm from the previous paragraph) in the space of all $L$-formulae.

The infinitary logic $L_{\omega_1,\omega}$ allows us to perform, in addition to the usual formation rules for describing formulae, two new operations, namely countable supremum $\bigvee$ and countable infimum $\bigwedge$.  However, in order to be able to form $\bigvee_m\varphi_m$ or $\bigwedge_m \varphi_m$, two things are required:  (1) all $\varphi_m$ have free variables among some fixed set $x=x_1,\ldots,x_n$ of variables; and (2)  the infimum of the moduli of uniform continuity of each $\varphi_m$ is itself a modulus of uniform continuity.  (See, for example, \cite[Definition 1.1]{BI} for more details.)

Throughout, $\u$ denotes an arbitrary nonprincipal ultrafilter on $\n$.  For an $L$-structure $A$, $A^\u$ denotes the ultrapower of $A$ with respect to $\u$.  While the isomorphism type of this structure often depends on $\u$, the use of such an ultrapower will not depend on $\u$.  For example, if $T$ is an $L$-theory, we say that a separable model $A$ of $T$ is \emph{locally universal for $T$} if every separable model of $T$ embeds into $A^\u$.  It is a standard fact that this notion does not depend on $\u$.

Recall that if $\theta:A\to B$ is an embedding between $L$-structures, then $\theta$ is said to be \emph{existential} if, for any quantifier-free $L$-formula $\phi(x,y)$  and any tuple $a$ from $A$, we have
$$\inf_{b\in A}\phi(a,b)=\inf_{b\in B}\phi(\theta(a),b).$$  If $A$ is a substructure of $B$ and the inclusion map is existential, we say that $A$ is \emph{existentially closed in} $B$.  An equivalent semantic reformulation of the latter property reads:  $A$ is existentially closed in $B$ if and only if there is an embedding of $B$ into $A^\u$ which restricts to the diagonal embedding of $A$ into $A^\u$.  (If $A$ and $B$ are nonseparable, then $\u$ may need to live on a larger index set.)  If $T$ is an $L$-theory and $A$ is a model of $T$, we say that $A$ is an existentially closed (or simply e.c.) model of $T$ if $A$ is existentially closed in all extensions that are models of $T$.  Then $A\models T$ is e.c. for $T$ if and only if $A$ is e.c. for $T_\forall$, where $T_\forall$ is the collection of closed conditions $\sigma=0$ such that $\sigma$ is universal and $T\models \sigma=0.$  Also, if $T$ has the \emph{joint embedding property} (JEP), namely that every pair of models of $T$ can be embedded into a common model of $T$, then existentially closed models of $T$ are locally universal for $T$.

A particularly important case is the case that $T$ is an $\forall\exists$-axiomatizable theory.    Then every (separable) model of $T$ embeds into a (separable) e.c. model of $T$. 

\section{Games and forcing}

\subsection{Introducing the game}

Until further notice, $L$ is a fixed countable continuous signature and $T$ is an $L$-theory.  For the sake of simplicity, we assume that our language is 1-sorted, bounded, and that each predicate (including the metric) takes values in $[0,1]$.  (Note that this is certainly not the case for the languages and theories applicable in operator algebras, but we trust that the reader should have no trouble convincing themselves that everything we do here can be adapted to the more general setting.)

We let $C$ be a countable set of new constant symbols and set $L(C):=L\cup C$.  Following the convention from \cite{hodges}, we denote $L(C)$-structures by $A^+$, $B^+$, etc... and the corresponding $L$-reducts by $A$, $B$, etc...  We call an $L(C)$-structure \emph{canonical} if the interpretations of the symbols from $C$ are dense; if, moreover, every open ball contains infinitely  many such interpretations, we call the structure \emph{extra canonical}.

A \emph{condition} is a finite set $p$ of expressions of the form $\varphi<r$, where $\varphi$ is a quantifier-free \emph{restricted} $L(C)$-sentence, such that $T\cup p$ is satisfiable.\footnote{We should probably call these conditions \emph{open conditions} to distinguish them from the conditions $\varphi=0$ used, for example, in \cite{bbhu}.  However, we hope that this poses no confusion.}

As mentioned in the introduction, the game involves two players, $\forall$ and $\exists$.  Players $\forall$ and $\exists$ take turns playing conditions extending the previous players move.  Thus, $\forall$ starts by playing the condition $p_0$, whence $\exists$ follows up by playing the condition $p_1\supseteq p_0$, and then $\forall$ follows that play with some condition $p_2\supseteq p_1$, etc... After $\omega$ many steps, the two players have together played a chain $p_0\subseteq p_1\subseteq p_2\subseteq\cdots$ of conditions whose union we will denote by $\bar p$.  

We call the above play \emph{definitive} if, for every atomic $L(C)$-sentence $\varphi$, there is a unique $r\in [0,1]$ such that $T\cup \bar p \models \varphi=r$.  In this case, $\bar p$ describes an $L(C)$-prestructure $A_0^+(\bar p)$ whose completion will be denoted by $A^+(\bar p)$ and will be referred to as the \emph{compiled structure}.\footnote{To wit:  the underlying universe of $A^+_0$ is the term algebra on the set of constants from $C$ and the symbols are interpreted in the obvious way.}  The reduct of $A^+(\bar p)$ to $L$ will be denoted by $A(\bar p)$.  If $\bar p$ is clear from context, we will denote $A^+(\bar p)$ and $A(\bar p)$ simply by $A^+$ and $A$ respectively.

Note that, regardless of player $\forall$'s moves, player $\exists$ can always ensure that the play of the game is definitive.

\begin{df}
Let $P$ be a property of $L(C)$-structures.  The game $G(P)$ is the game whose moves are as above and such that Player $\exists$ wins $G(P)$ if and only if $\bar p$ is definitive and $A^+(\bar p)$ has property $P$.  We say that $P$ is \emph{enforceable} if Player $\exists$ has a winning strategy in $G(P)$. 
\end{df}    

By the remark preceding this definition, the vacuously true property is enforceable.  We leave the proof of the following lemma to the reader.

\begin{lemma}
The property ``the compiled structure is extra canonical'' is enforceable.
\end{lemma}

While some properties may not be enforceable, they may become enforceable if the game has reached a certain point.

\begin{df}
Let $P$ be a property of $L(C)$-structures and let $p$ be a condition.  We say that $p$ \emph{forces} $P$ if, for any position $(p_0,\ldots,p_k)$  of the game $G(P)$, if $p\subseteq p_k$, then the position is winning for $\exists$.
\end{df}

The proof of the following lemma in the classical setting can be found in \cite{hodges}; the corresponding facts in the continuous setting provide no added difficulty.

\begin{lemma}\label{properties}

\

\begin{enumerate}
\item $p$ forces $P$ if and only if whenever $\forall$ plays $p_0\supseteq p$, then $p_0$ is a winning position for $\exists$.
\item $P$ is enforceable if and only if every condition forces $P$.
\item If $p$ forces $P$ and $q\supseteq p$, then $q$ forces $P$.
\item (Conjunction Lemma)  If $p$ forces $P_i$ for each $i<\omega$, then $p$ forces the conjunction of the $P_i$'s.
\end{enumerate}
\end{lemma}

\begin{prop}\label{enforceuniversal}
It is enforceable that the compiled structure be a model of $T_\forall$.
\end{prop}

\begin{proof}
Suppose that $\sigma=0$ belongs to $T_\forall$ with $\sigma=\sup_x \varphi(x)$, $\varphi(x)$ quantifier-free.  Let $c$ be a tuple of distinct constants and $n\geq 1$.  Since being an extra canonical structure is enforceable, by the conjunction lemma, it is enough to enforce that $\varphi(c)<1/n$.  Suppose that player $\forall$ plays $p_0$.  Let $A^+$ be a model of $T\cup p_0$; since $\sigma^{A^+}=0$, we know that $\varphi^{A^+}(c^{A^+})=0$.  Let $\psi(x)$ be a restricted quantifier-free formula such that $\|\varphi-\psi\|<\frac{1}{2n}$.  It follows that $p_1:=p_0\cup \{\psi(c)<\frac{1}{2n}\}$ is a condition.  If player $\exists$ plays $p_1$, then the compiled structure will be as desired.
\end{proof}

Call a sentence $\sigma$ of $L_{\omega_1,\omega}$ a \emph{$\sup\bigvee\inf$-sentence} if 
$$\sigma=\sup_x\inf_n \varphi_n(x),$$ where $x$ is a (finite) tuple of variables and each $\varphi_n$ is existential.  We call a property $P$ of $L$-structures a \emph{$\sup\bigvee\inf$-property} if there are $\sup\bigvee\inf$-sentences $\sigma_m$ such that an $L$-structure $A$ has property $P$ precisely when $\sigma_m^A=0$ for all $m$.  We can often enforce $\sup\bigvee\inf$-properties.

\begin{prop}\label{enforcesbvi}
Suppose that $P$ is a $\sup\bigvee\inf$-property.  Further suppose that there is a locally universal model of $T$ with property $P$.  Then $P$ is enforceable.
\end{prop}

\begin{proof}
Suppose that $A$ is a locally universal model of $T$ with property $P$.  Suppose that $P$ is defined by the $\sup\bigvee\inf$-sentences $$\sigma_m=\sup_{x_m}\inf_n \varphi_{mn}(x_m).$$  Fix a tuple of distinct constants $c$ and $k\geq 1$.  By the conjunction lemma and the fact that being extra canonical is enforceable, it suffices to show that $\inf_n\varphi_{mn}(c)<1/k$ is enforceable.  Here is the strategy:  suppose that $\forall$ opens with the condition $p_0$ which is satisfied in some $L(C)$-structure $B^+\models T$.  Embed $B:=B^+|L$ into an ultrapower $A^\u$ of $A$ and let $(A^\u)^+$ be the expansion of $A^\u$ to an $L(C)$-structure that makes this embedding of $L$-structures an embedding of $L(C)$-structures as well.  It follows that there is an expansion $A^+$ of $A$ such that $p_0$ is also satisfied in $A^+$.  Since $A$ has property $P$, there is $n$ such that $\varphi_{mn}(c^{A^+})<1/k$.  It follows that $p_1:=p_0\cup \{\varphi_{mn}(c)<1/k\}$ is a condition and if $\exists$ plays $p_1$, the compiled structure will satisfy the property $\inf_n \varphi_{mn}(c)<1/k$.
\end{proof}

Of particular interest is what (infinitary) first-order properties can be forced. 
\begin{df}
Let $p$ be a condition, $\varphi$ an $L(C)_{\omega_1,\omega}$-sentence, and $r\in \r^{>0}$.  We write $p\Vdash^g \varphi<r$ if $p$ forces the property $\varphi<r$.\footnote{We use the notation $\Vdash^g$ to indicate that this is the forcing property stemming naturally from the game apparatus.  In the next section, we will soon see that this is exactly the notion of \emph{weak forcing} already present in the literature.}  When $p=\emptyset$, we simply write $\Vdash^g \varphi<r$.  We also set $$F_p^g(\varphi):=\inf\{r \ : \ p\Vdash^g \varphi<r\}.$$
\end{df}

By Lemma \ref{properties} (3), we have that $q\supseteq p$ implies $F_q^g(\varphi)\leq F_p^g(\varphi)$.  The following lemma is immediate from the definitions.

\begin{lemma}\label{forceapprox}
Suppose that $p\Vdash^g \varphi<r$ and $\|\varphi-\psi\|<\epsilon$.  Then $p\Vdash^g \psi<r+\epsilon$.
\end{lemma}

The following lemma is quite useful:

\begin{lemma}\label{existentialforcing}
Suppose that $p$ is a condition all of whose constants are contained in the tuple $c$.  Further suppose that $\theta(x)$ an existential $L$-formula and $\epsilon>0$ are such that $T\cup p\cup \{\theta(c)<\epsilon\}$ is satisfiable.  Then there is a condition $q\supseteq p$ such that $q\Vdash^g \theta(c)<\epsilon$.
\end{lemma}

\begin{proof}
Write $\theta(x)=\inf_y \psi(x,y)$.  Let $B^+$ be an $L(C)$-structure such that $B^+\models T\cup p\cup \{\theta(c)<\epsilon\}$.  Let $d$ be a tuple of constants such that $\psi(c,d)<\epsilon$.  Then $q:=p\cup\{\psi(c,d)<\epsilon\}$ is a condition and clearly $q\Vdash^g \theta(c)<\epsilon$.
\end{proof}

The following proposition is central for much of what we do in future sections.

\begin{prop}
It is enforceable that the compiled structure be an e.c.\ model of $T_\forall$.  
\end{prop}

\begin{proof}
Suppose that $c$ is a tuple of distinct constants, $\varphi(x)$ is an existential formula, and $r\in \mathbb{Q}^{>0}$.  By the Conjunction Lemma, it is enough to enforce the following property:  if $\varphi(c)\geq r$, then there is no extension of the compiled structure that models $T_\forall \cup \{\varphi(c)<r\}$.  Here is the winning strategy for $\exists$:  Suppose that $\forall$ plays $p_0=\{\psi_i(c,d)<r_i \ : \ i=1,\ldots,k\}$. If $T\cup p_0\models \varphi(c)\geq r$, then $$T_\forall\models \sup_x\min(\sup_y \min_{1\leq i\leq k} r_i\dotminus \psi_i(x,y),r\dotminus \varphi(x))=0.$$  In this case, no extension of of the compiled structure models $T_\forall \cup\{\varphi(c)<r\}$, whence the conditional statement that we are trying to enforce is true.  Otherwise, by Lemma \ref{existentialforcing}, there is $B^+\models T\cup p_0\cup \{\varphi(c)<r\}$, so there is a constant $c'$ such that $p_1:=p_0\cup \{\theta(c,c')<r\}$ is a condition, where $\varphi(x)=\inf_y \theta(x,y)$.  If $\exists$ play $p_1$, then the compiled structure models $\varphi(c)<r$, whence, once again, the conditional statement that we are trying to enforce is true.    
\end{proof}

Our notion of forcing satisfies a useful homogeneity property.  For $\pi$ a permutation of $C$ and $\varphi$ an $L(C)_{\omega_1,\omega}$-sentence, let $\pi(\varphi)$ be the $L(C)_{\omega_1,\omega}$-sentence obtained by replacing every $c\in C$ with $\pi(c)$.  If $p$ is a condition, let $\pi(p)$ denote the condition obtained by replacing every $\varphi<r$ in $p$ with $\pi(\varphi)<r$.  Once again, we leave the proof of the next lemma to the reader.  

\begin{lemma}[Homogeneity]\label{homog}
Suppose that $\pi$ is a permutation of $C$, $p$ a condition, $\varphi$ an $L(C)_{\omega_1,\omega}$-sentence, and $r\in \r^{>0}$.  If $p\Vdash^g \varphi<r$, then $\pi(p)\Vdash^g \pi(\varphi)<r$.
\end{lemma}

The next lemma is the analog of \cite[Lemma 2.3.3(d)]{hodges} and is an indication of why game forcing coincides with weak forcing.

\begin{lemma}\label{twostep}
For every condition $p$ and every $L(C)_{\omega_1,\omega}$-sentence $\varphi$, we have
$$F_p^g(\varphi)=\sup_{q\supseteq p}\inf_{q'\supseteq q}F_{q'}^g(\varphi).$$
\end{lemma}

\begin{proof}
First, if $q\supseteq p$, then $\inf_{q'\supseteq q} F_{q'}^g(\varphi)\leq F_q^g(\varphi)\leq F_p^g(\varphi)$, so $$\sup_{q\supseteq p}\inf_{q'\supseteq q} F_{q'}^g(\varphi)\leq F_p^g(\varphi).$$  Now suppose that $\sup_{q\supseteq p}\inf_{q'\supseteq q} F_{q'}^g(\varphi)<r$; it suffices to show that $p\Vdash^g \varphi< r$.  Here is the strategy:  suppose that $\forall$ opens with $q\supseteq p$.  Then $\exists$ should play $q'\supseteq q$ so that $F_{q'}^g(\varphi)<r$, for then $q'\Vdash^g \varphi<r$ and by following the strategy that witnesses this latter statement, $\exists$ can enforce that $\varphi<r$ holds at the end of this play.  
\end{proof}

\subsection{The finite forcing companion $T^f$}

In this subsection, we discuss the finite forcing companion $T^f$ of $T$ consisting of all enforceable conditions.  The main result that we want to establish is that the theory $T^f$ is complete when $T$ has JEP.  Towards this end, we first show that, given any sentence $\varphi$, any ``position'' $p$, and any ``accuracy'' $\epsilon>0$, we can find a further position $q\supseteq p$ forcing $\varphi$ to have a value in an interval of length at most $\epsilon$.  It will become useful to extend our official use of the symbol $\Vdash^g$.  For example, we may write $p\Vdash^g a<\varphi<b$ to mean that $p$ forces the property $a<\varphi<b$.  This result is similar to \cite[Remark 3.6]{farahmagidor}.

\begin{prop}
Given a condition $p$, $L(C)_{\omega_1,\omega}$-sentence $\varphi$, and $\epsilon>0$, there is $q\supseteq p$ and $a<b$ with $b-a<\epsilon$ such that $q\Vdash^g a<\varphi<b$.
\end{prop}

\begin{proof}
By Lemma \ref{forceapprox}, it suffices to consider the case that $\varphi$ is restricted.  We may thus prove the proposition for such $\varphi$ by induction on complexity.  

First suppose that $\varphi$ is atomic.  Since $p$ is a condition, there is $A^+\models T\cup p$.  Set $r:=\varphi^{A^+}$ and set $q:=p\cup \{|\varphi-r|<\epsilon\}$.  Then $q$ is a condition extending $p$ and $q\Vdash^g |\varphi-r|<\epsilon$.

If $\varphi=\neg \psi$ and $q\supseteq p$ is such that $q\Vdash^g a<\psi<b$ with $b-a<\epsilon$, then $q\Vdash^g 1-b<\varphi<1-a$.  The case that $\varphi=\frac{1}{2}\psi$ is handled similarly.

Now suppose that $\varphi=\psi\dotplus\theta$.  Take $q'\supseteq p$ such that $q'\Vdash^g a<\psi<b$ with $b-a<\frac{\epsilon}{2}$.  Take $q\supseteq q'$ such that $q\Vdash^g c<\theta<d$ with $d-c<\frac{\epsilon}{2}$.  Then $q\Vdash^g a+c<\varphi<b+d$ and $(b+d)-(a+c)<\epsilon$.

Now suppose that $\varphi=\bigwedge \phi_i$.  Let $a:=\inf\{a' \ : q\Vdash^g \varphi<a' \text{ for some }q\supseteq p\}$.  Set $\delta:=\frac{\epsilon}{3}$ and take $q\supseteq p$ such that $q\Vdash^g \varphi<a+\delta$.  We claim that $q\Vdash^g \varphi\geq a-2\delta$, which settles this case.  To establish this claim, we use Lemma \ref{twostep}.  Take $q'\supseteq q$.  For each $i$, take $q_i\supseteq q'$ such that $q_i\Vdash^g c_i<\varphi_i<d_i$ with $d_i-c_i<\epsilon$.  Since $q_i\not\Vdash^g \varphi_i<a-\epsilon$ (else $q_i\Vdash^g \varphi<a-\epsilon$, contradicting the definition of $a$), we have $a-\epsilon<d_i$.  It follows that $q_i\Vdash^g \varphi_i>c_i>d_i-\epsilon>a-2\epsilon$.  Since $q'\supseteq q$ was arbitrary, it follows that $q\Vdash^g \varphi_i>a-2\epsilon$ for each $i$, whence $q\Vdash^g \varphi\geq a-2\epsilon$.

Finally, suppose that $\varphi=\inf_x \psi(x)$.  By the previous case, we may choose $q\supseteq p$ such that $q\Vdash^g a<\inf_{c\in C}\psi(c)<b$ with $b-a<\frac{\epsilon}{2}$.  Since being canonical is enforceable, it follows that $q\Vdash^g a-\frac{\epsilon}{2}< \inf_x \psi(x)<b$.
\end{proof}

\begin{thm}
Suppose that $T$ has JEP.  Then for every $L_{\omega_1,\omega}$-sentence $\varphi$, there is a unique $r$ such that $\varphi=r$ is enforceable.
\end{thm}

\begin{proof}
Fix $\epsilon>0$.  Fix an interval $(a,b)$ of length less than $\epsilon$ and a condition $p$ such that $p\Vdash^g a<\varphi<b$.  We claim that $\Vdash^g a\leq \varphi\leq b$.  Suppose otherwise.  Without loss of generality, we may assume that $\not\Vdash^g \varphi\leq b$.  Take $\delta>0$ such that $\not\Vdash^g \varphi<b+\delta$.  Take $(c,d)$ with $d-c<\delta$ and $q$ such that $q\Vdash^g c<\varphi<d$.  Then $b+\delta<d$ so $b<d-\delta<c$, so $(a,b)\cap (c,d)=\emptyset$.  Since $\varphi$ has no constants from $C$, by Lemma \ref{homog}, we may assume that $p$ and $q$ have no constants in common and thus can be realized in a common model of $T$ by JEP, which is a contradiction as then $p\cup q$ is a condition and $p\cup q\Vdash^g \varphi\in (a,b)\cap (c,d)$.  

Taking $\epsilon=\frac{1}{n}$, we get intervals $(a_n,b_n)$ of length at most $\frac{1}{n}$ such that $\Vdash^g a_n\leq \varphi\leq b_n$.  If $\bigcap_n [a_n,b_n]=\{r\}$, the Conjunction Lemma implies that $\varphi=r$ is enforceable.
\end{proof}

\begin{df}
We let $T^f$ be the $L$-theory containing the closed conditions $\sigma=r$ whenever that condition is enforceable.  $T^f$ is called the \emph{finite forcing companion of $T$}.
\end{df}

\begin{cor}
If $T$ has JEP, then $T^f$ is complete.
\end{cor}

Given an $L$-structure $A$, we may always expand it to a canonical $L(C)$-structure $A^+$ (although there is no canonical choice for doing this).  We may then define the \emph{diagram of $A$} to be the set of closed $L(C)$-conditions of the form $\varphi=0$, where $\varphi$ is a quantifier-free $L(C)$-sentence such that $\varphi^{A^+}=0$.  It is then a standard fact that $B^+\models T\cup \Diag(A)$ if and only if $B\models T$ and $A$ embeds into $B$.  (Even though $\Diag(A)$ depends on how we expand $A$ to $A^+$, this latter fact is independent of our choice.)  We will also write $\appdiag(A^+)$ (or simply $\appdiag(A)$) for the set of expressions of the form $\varphi<r$, where $\varphi$ is a restricted quantifier-free $L(C)$-sentence such that $\varphi^A<r$.  Of course, $T\cup \Diag(A)$ and $T\cup \appdiag(A)$ have the same models.  If $A$ happens to be a model of $T$, then $\appdiag(A)$ is the union of the set of conditions that are satisfied in $A^+$.

\begin{cor}\label{mutual}
$T_\forall=(T^f)_\forall$.
\end{cor}

\begin{proof}
Since it is enforceable that the compiled structure is a model of $T_\forall$, we have that $T_\forall\subseteq T^f$.  For the other direction, it is enough to show that any model of $T_\forall$ can be extended to a model of $T^f$.  Suppose that $A\models T_\forall$; we need to show that  $\Diag(A)\cup T^f$ is satisfiable.  By compactness, it suffices to show that, for any $p\in \appdiag(A)$ and any condition $\sigma=0$ belonging to $T^f$, there is a model of $p\cup \{\sigma=0\}$.  Since $\sigma=0$ is enforceable, we have that $p\Vdash^g\sigma$, whence by following the strategy we can construct a model where $p$ holds and $\sigma=0$ is true.
\end{proof}

\subsection{Locally universal models revisited}

As pointed out in the introduction, if $T$ has JEP and $A$ is an e.c. model of $T$, then $A$ is a locally universal model of $T$.  Since being an e.c. model of $T_\forall$ is enforceable, it follows that if $T$ has JEP and $B$ is a separable model of $T$ such that being $B^\u$-embeddable is enforceable, then $B$ is a locally universal model of $T$.  However, it turns out that the conclusion of the following sentence is true even without assuming JEP.

\begin{thm}\label{enforcelocallyuniversal}
Suppose that $B$ is a model of $T$ such that being $B^\u$-embeddable is enforceable.  Then $B$ is a locally universal model of $T$.
\end{thm}

\begin{proof}
Suppose that $A$ is a separable model of $T$.  By saturation, it suffices to show, given any condition $p\subseteq \appdiag(A)$, that $B^\u$ has an expansion to an $L(C)$-structure that is a model of $p$.  Viewing $p$ as $\forall$'s first move, by following $\exists$'s strategy to ensure that the compiled structure is $B^\u$-embeddable, it follows that $p$ can be satisfied in $B^\u$, as desired.
\end{proof}

\subsection{Connection to weak forcing}

Model-theoretic forcing has already appeared in continuous logic in many places, the first being \cite{BI}.  The purpose of this section is to connect the above forcing theory with that already appearing in the literature.  As alluded to in \cite[Historical Reference for Chapter 2]{hodges}, the forcing associated with games is the same as what is traditionally referred to as weak forcing.\footnote{Unfortunately, this helpful remark is quite hidden in this section.  In fact, Hodges simply writes ``(Our forcing is weak.)''}  The purpose of this subsection is to show that, in fact, the function $F_p^g$ defined above coincides with the corresponding function $F_p^w$ for weak forcing appearing in \cite{BI}.

Let us first review the setup from \cite{BI}.  If $p$ is a condition and $\varphi$ is a restricted atomic $L(C)$-sentence, we define $f_p(\varphi):=\min\{r\leq 1 \ | \varphi<r\in p\}$, with the understanding that $\min(\emptyset)=1$.  For a condition $p$ and a restricted $L(C)_{\omega_1,\omega}$-sentence $\varphi$, we define the value $F_p(\varphi)\in [0,1]$ by induction on $\varphi$:
\begin{itemize}
\item $F_p(\varphi)=f_p(\varphi)$ if $\varphi$ is atomic.
\item $F_p(\neg \varphi)=\neg \inf_{q\supseteq p}F_q(\varphi)$.
\item $F_p(\frac{1}{2}\varphi)=\frac{1}{2}F_p(\varphi)$.
\item $F_p(\varphi\dotplus \psi)=F_p(\varphi)\dotplus F_p(\psi)$.
\item $F_p(\bigvee \Phi)=\inf_{\varphi\in \Phi}F_p(\varphi)$.
\item $F_p(\inf_x\varphi(x))=\inf_{c\in C}F_p(\varphi(c))$.
\end{itemize}

If $r\in \r$ and $F_p(\varphi)<r$, we say that $p$ \emph{(strongly) forces} that $\varphi<r$, and write $p\Vdash \varphi<r$.

We can now define the weak forcing relation.

\begin{df}
For a condition $p$ and a restricted $L(C)_{\omega_1,\omega}$-sentence $\varphi$, we set $$F_p^w(\varphi)=\sup_{q\supseteq p}\inf_{q'\supseteq q}F_{q'}(\varphi).$$  If $r\in \r$ and $F_p^w(\varphi)<r$, we say that $p$ \emph{weakly forces} that $\varphi<r$, and write $p\Vdash^w \varphi<r$.
\end{df}

The following facts are Lemma 2.8 and Proposition 2.9 from \cite{BI} respectively.

\begin{fact}
For a condition $p$ and a restricted $L(C)_{\omega_1,\omega}$-sentence $\varphi$, we have $$F_p^w(\varphi)=\sup_{q\supseteq p}\inf_{q'\supseteq q}F_{q'}^w(\varphi).$$  
\end{fact}

\begin{fact}
$F_p^w$ satisfies the following inductive rules.
\begin{itemize}
\item $F_p^w(\neg \varphi)=\neg \inf_{q\supseteq p}F_q^w(\varphi)$.
\item $F_p^w(\frac{1}{2}\varphi)=\frac{1}{2}F_q^w(\varphi)$.
\item $F_p^w(\varphi\dotplus\psi)=\sup_{q\supseteq p}\inf_{q'\supseteq q}F_{q'}^w(\varphi)\dotplus F_{q'}^w(\psi)$.
\item $F_p^w(\bigvee \Phi)=\sup_{q\supseteq p}\inf_{q'\supseteq q}\inf_{\varphi\in \Phi}F_{q'}^w(\varphi)$.
\item $F_p^w(\inf_x\varphi(x))=\sup_{q\supseteq p}\inf_{q'\supseteq q}\inf_{c\in C}F_{q'}^w(\varphi(c))$.
\end{itemize}
\end{fact}

The following is the main result of this subsection; it says that game forcing and weak forcing are the same.
\begin{thm}\label{coincides}
For all conditions $p$ and restricted $L(C)_{\omega_1,\omega}$-sentences $\varphi$, we have $F_p^g(\varphi)=F_p^w(\varphi)$.
\end{thm}

\begin{proof}
We proceed by induction on the complexity of $\varphi$.

First suppose that $\varphi$ is atomic.  Fix $\epsilon>0$ and choose $r$ such that $r<F_p^g(\varphi)+\epsilon$ and $p\Vdash^g \varphi<r$.  Fix $q\supseteq p$, so $q\Vdash^g \varphi<r$.  In particular there is $s<r$ such that $q':=q\cup \{\varphi<s\}$ is consistent.  Clearly $q'\Vdash^w \varphi<s$ so $F_p^w(\varphi)\leq r<F_p^g(\varphi)+\epsilon$; letting $\epsilon$ approach $0$, we see that $F_p^w(\varphi)\leq F_p^g(\varphi)$.  Conversely, if $p\Vdash^w \varphi<r$, then for every $q\supseteq p$, there is $q'\supseteq q$ such that $q'\Vdash \varphi<r$, whence $\varphi<s$ belongs to $q'$ for some $s<r$ and thus $q'\Vdash^g \varphi<s$ and $F_{q'}^g(\varphi)<s$.  By Lemma \ref{twostep}, we have that $F_p^g(\varphi)\leq r$ and thus $F_p^g(\varphi)\leq F_p^w(\varphi)$.

Now suppose that $\varphi=\neg \psi$.  First suppose that $p\Vdash^g \varphi<r$.  Take $q\supseteq p$.  Then $q\not\Vdash^g \psi<1-r$, so $F_q^g(\psi)\geq 1-r$ whence $F_q^w(\psi)\geq 1-r$ by induction.  Therefore, 
$$F_p^w(\varphi)=1\dotminus \inf_{q\supseteq p}F_q^w(\psi)\leq 1-(1-r)=r$$ and thus $F_p^w(\varphi)\leq F_p^g(\varphi)$.  Now suppose that $p\Vdash^w \varphi<r$ and fix $q\supseteq p$ and $\epsilon>0$; it suffices to find $q'\supseteq q$ such that $q'\Vdash^g \varphi<r+\epsilon$, for then, by Lemma \ref{twostep}, we have that $F_p^g(\varphi)\leq r+\epsilon$ and thus, letting $\epsilon$ approach $0$, we have that $F_p^g(\varphi)\leq F_p^w(\varphi)$.  Take $q'\supseteq q$ such that $q'\Vdash^g a<\psi<b$ with $b-a<\epsilon$.  Since $p\Vdash^w \varphi<r$, we have that $F_{q'}^w(\psi)\geq 1-r$, whence $1-r<b$ by the induction hypothesis.  It follows that $q'\Vdash^g \varphi<1- a<1-b+\epsilon<r+\epsilon$, as desired.

The case that $\varphi=\frac{1}{2}\psi$ is easy.  Now suppose that $\varphi=\psi\dotplus\theta$.  We first show that $F_p^w(\varphi)\leq F_p^g(\varphi)$.  Suppose $p\Vdash^g \varphi<r$.  Take $q\supseteq p$; it suffices to show that $\inf_{q'\supseteq q}(F_{q'}^w(\psi)\dotplus F_{q'}^w(\theta))\leq r$.  Fix $\epsilon>0$ and take $q'\supseteq q$ such that $q'\Vdash^g a-\epsilon<\psi<a+\epsilon$ and $q'\Vdash^g b-\epsilon<\theta<b+\epsilon$.  It follows that $a+b-2\epsilon<r$ and that (by induction) $F_{q'}^w(\psi)\leq a+\epsilon$ and $F_{q'}^w(\theta)\leq b+\epsilon$, so $F_{q'}^w(\psi)\dotplus F_{q'}^w(\theta)\leq a+b+2\epsilon<r+4\epsilon$; letting $\epsilon$ approach $0$ yields the desired result.  Now suppose that $p\Vdash^w \varphi<r$.  Fix $q\supseteq p$.  Take $q'\supseteq q$ such that (by induction) $F_{q'}^g(\psi)\dotplus F_{q'}^g(\theta)<r$.  Then there are $s,t$ such that $q'\Vdash^g \psi<s$ and $q'\Vdash^g \theta<t$ and $s+t<r$.  It follows that $q'\Vdash^g \varphi<r$, whence $F_p^g(\varphi)<r$.

Now suppose that $\varphi=\bigvee \varphi_i$.  First suppose that $F_p^w(\varphi)<r$.  Fix $q\supseteq p$ and find $q'\supseteq q$ and $i$ such that $F_{q'}^w(\varphi_i)<r$, whence $F_{q'}^g(\varphi_i)<r$ and thus $q'\Vdash^g \varphi_i<r$ and hence $q'\Vdash^g \varphi<r$.  It follows that $F_p^g(\varphi)\leq F_p^w(\varphi)$.  Now suppose that $F_p^g(\varphi)<r$.  Fix $q\supseteq p$.  Then there is $i$ such that $q\not\Vdash^g \neg\varphi_i\leq 1-r$, whence $q\not\Vdash^w \neg \varphi_i\leq 1-r$, and thus there is $q'\supseteq q$ such that $q'\Vdash^w \varphi_i<r$.  It follows that $p\Vdash^w \varphi<r$.

Finally suppose that $\varphi=\inf_x \psi(x)$.  Since it is enforceable that the compiled structure is canonical, we have that $p\Vdash^g \varphi<r$ if and only if $p\Vdash^g \bigvee_{c\in C}\psi(c)<r$; now use the previous case.
\end{proof}

\section{Finite-generic structures}

In this section, we once again fix an $L$-theory $T$ and forcing is with respect to this theory.  

\subsection{Introducing finite-generic structures}  Suppose that $B^+$ is a canonical $L(C)$-structure.  Given an $L(C)_{\omega_1,\omega}$-sentence $\varphi$ and $r\in \r^{>0}$, we say $B^+$ forces $\varphi<r$, written $B^+\Vdash^g \varphi<r$, if there is a condition $p\subseteq \appdiag(B^+)$ with $p\Vdash^g \varphi<r$.  In general, whether a structure forces the expression $\varphi<r$ (for $\varphi$ finitary) is not the same as whether or not $\varphi^{B^+}<r$ is true.  These notions coincide for a very important class of structures: 

\begin{df}
We say that a canonical $L(C)^+$-structure is \emph{finite-generic$^+$} if, for any (finitary) $L(C)$-sentence $\varphi$ and any $r\in \r^{>0}$, we have
$$B^+\Vdash^g \varphi<r \Leftrightarrow B^+\models \varphi<r.$$
\end{df}

We leave the following lemma to the reader:

\begin{lemma}
Suppose that $B^+$ is a canonical $L(C)$-structure.  Then $B^+$ is finite-generic$^+$ if and only if for every $L(C)$-sentence $\sigma$ and every $\epsilon>0$, if $\sigma^{B^+}=r$, then $B^+\Vdash^g |\sigma-r|<\epsilon$.
\end{lemma}


\begin{df}
An $L$-structure will be called \emph{finite-generic} if it is the $L$-reduct of a finite-generic$^+$ $L(C)$-structure.
\end{df}

If we want to emphasize the base theory $T$, we shall say that $B$ is finite-generic with respect to $T$.

\begin{rmk}
If one compares our definition of finite-generic$^+$ with the corresponding classical definition (\cite[Section 4.3]{hodges}), it seems as if we should demand that finite-generic$^+$ structures be \emph{extra canonical}.  It does appear that this leads to a more restrictive notion of finite-generic$^+$-structures, but we invite the reader to check that this leaves the class of reducts (i.e. finite-generic structures) unchanged.  (Simply add a new set of countably many constants.)  It seems that Hodges prefers the more restrictive notion to make certain proofs easier but we note that this is not at all necessary.
\end{rmk}

Let us next show how this notion is the same as the one presented in the continuous logic literature using generic sets of conditions.

\begin{df}
Let $G$ be a nonempty set of conditions.  We say that $G$ is \emph{generic} if:
\begin{itemize}
\item the union of two elements of $G$ is once again an element of $G$, and 
\item for every restricted $L(C)$-sentence $\varphi$ and every $r>1$, there is $p\in G$ such that $F_p(\varphi)+F_p(\neg \varphi)<r$.
\end{itemize}
\end{df}

It is proven in \cite{BI} that generic sets always exist.  If $G$ is generic and $\varphi$ is a restricted $L(C)$-sentence, set $\varphi^G:=\inf_{p\in G}F_p(\varphi)$.  \cite[Lemma 2.13]{BI} asserts that $\varphi^G=\inf_{p\in G}F_p^w(\varphi)$.

The following fact combines Lemma 2.16 and Theorem 2.17 from \cite{BI}.

\begin{fact}[Generic Model Theorem]
Let $M_0^G$ denote the term algebra equipped with the natural interpretation of the function symbols and interpreting the predicate symbols by $P^{M_0^G}(\vec \tau):=P(\vec \tau)^G$.  Let $M^G$ be the completion of $M_0^G$.  Then $M^G$ is a canonical $L(C)$-structure such that, for all restricted $L(C)$-sentences $\varphi$, we have $\varphi^{M^G}=\varphi^G$.
\end{fact}

We can now prove:

\begin{prop}
$B^+$ is finite-generic$^+$ if and only if $B^+\cong M^G$ for some generic filter $G$.
\end{prop}

\begin{proof}
First suppose that $B^+$ is finite-generic.  Let $G$ consist of all conditions contained in $\appdiag(B^+)$.  We first note that $G$ is generic.  It is clear that the union of two conditions in $G$ is a condition in $G$ again.  Now suppose that $\varphi$ is an $L(C)$-sentence and $r>1$.  Choose $\epsilon>0$ such that $1+2\epsilon<r$.  Take $p\in G$ and $a\in \r$ such that $p\Vdash^g |\varphi-a|<\epsilon$.  Then $F_p^w(\varphi)+F_p^w(\neg \varphi)<1+2\epsilon<r$.  It follows that $G$ is generic.  By the construction of $M^G$, it is now clear that $B^+\cong M^G$.

Conversely, suppose that $G$ is a generic set; we must show that $M^G$ is finite-generic$^+$.  Suppose that $\varphi$ is a restricted $L(C)$-sentence such that $\varphi^{M^G}=r$.  Fix $\epsilon>0$.  Since $\varphi^G=r$, we have $p\in G$ such that $p\Vdash^w \varphi<r+\epsilon$; since $p\subseteq \appdiag(M^G)$, by Theorem \ref{coincides}, we have $M^G\Vdash^g \varphi<r+\epsilon$.  By considering $\neg \varphi$, we see that $M^G\Vdash^g |\sigma-r|<\epsilon$.  By density of the restricted formulae, we see that $M^G$ is finite-generic$^+$.
\end{proof}

\begin{rmk}
Returning to our earlier remarks about the difference between our definition of finite-generic$^+$ and the one appearing in \cite{hodges}, we note that if one were to demand finite-generic$^+$-structures to be extra canonical, then it does not appear that one would be able to obtain the previous proposition as the generic $G$ may be agnostic about the infinitary expression $\bigvee_{k>m} d(c_k,c_n)<\epsilon$.
\end{rmk}

\begin{prop}\label{enforcegeneric}
 Being finite-generic$^+$ is enforceable.
\end{prop}

\begin{proof}
We already know that we can enforce the compiled structure to be canonical.  Now suppose that $\varphi$ is a restricted $L(C)$-sentence and $\epsilon\in \mathbb{Q}^{>0}$.  It suffices to show that we can enforce the following property:  if $\sigma^{A^+}=r$, then $A^+\Vdash^g |\varphi-r|<\epsilon$.  Here is the strategy:  suppose that $\forall$ plays $p_0$.  Then there is $p_1\supseteq p_0$ and an interval $I$ with $|I|<\epsilon$ such that $p_1\Vdash^g \varphi\in I$.  Have $\exists$ play $p_1$ and use the winning strategy.  Then the compiled structure $A^+$ will have $p_1\subseteq \appdiag(A^+)$.
\end{proof}

The following characterization of finite-generic structures in terms of the forcing companion $T^f$ will prove quite useful. 

\begin{prop}\label{chargen}
For an $L$-structure $B$, the following are equivalent:
\begin{enumerate}
\item $B$ is finite-generic;
\item $B\models T^f$ and for all $B\subseteq C\models T^f$, we have $B\preceq C$;
\item $B\models T_\forall$ and for all $B\subseteq C\models T^f$, we have $B\preceq C$.
\end{enumerate}
\end{prop}

\begin{proof}
(1) implies (2):  Suppose that $B$ is the reduct of the finite-generic$^+$ structure $B^+$.  We first show that $B\models T^f$.  Suppose that $\varphi^{T^f}=r$ but $\varphi^B=s\not=r$.  Fix $\epsilon>0$ such that $|r-s|\geq \epsilon$.  Then we arrive at a contradiction since $B^+\Vdash^g |\varphi-s|<\epsilon$ whilst $\Vdash^g |\varphi-r|<\epsilon$.  Thus $B\models T^f$.  

Now suppose that $B\subseteq C\models T^f$.  Let $\varphi(c)$ be an $L(C)$-sentence and $r\in \r^{>0}$ such that $\varphi(c)^{B^+}<r$; it suffices to show that $\varphi(c^{B^+})^C\leq r$.  Take $p\subseteq \appdiag(B^+)$ such that $p\Vdash^g \varphi(c)<r$.  Write $p=\{\psi_i(c,d)<r_i \ : \ i=1,\ldots,k\}$ where $d$ is a tuple of distinct constants disjoint from the tuple $c$.  Then, for any $\epsilon>0$, we have that $\min(\min_{1\leq i\leq k} (r_i\dotminus\psi_i(c,d)),\varphi(c)\dotminus r)=0$ is enforceable.  Indeed, if player $\forall$ plays $p_0$, then either $p_0\cup p$ is unsatisfiable (whence the first term in the minimum is $0$ in the compiled structure) or else $\exists$ can play $p_0\cup p$ and then follow the strategy witnessing $p\Vdash^g \varphi(c)<r$.  By homogeneity and the fact that being extra canonical is enforceable, it follows that the closed condition $$(\sup_x\sup_y\min(\min_i (r_i\dotminus\psi_i(x,y)),\varphi(x)\dotminus r))=0$$ is enforceable, whence belongs to $T^f$.  Since $$r_i\dotminus\psi_i(c^{B^+},d^{B^+})^C=r_i\dotminus \psi_i(c^{B^+},d^{B^+})^B$$ and $C\models T^f$, we have $\varphi(c^{B^+})^C\leq r$.

(2) implies (3) follows from the fact that $T_\forall\subseteq T^f$.  Now suppose that (3) holds.  Expand $B$ to a canonical $L(C)$-structure $B^+$.  Now suppose that $\varphi$ is an $L(C)$-sentence such that $\varphi^{B^+}=r$ and fix $\epsilon>0$.  By (3), $\Diag(B^+)\cup T^f\models |\varphi-r|<\epsilon$.  By compactness, there is $p\subseteq \appdiag(B^+)$ and a closed condition $\chi=0$ from $T^f$ such that $p\cup\{\chi=0\}\models |\varphi-r|<\epsilon$.  It suffices to show that $p$ is a condition, for then since $\chi=0$ is enforceable, we have that $p\Vdash^g |\varphi-r|<\epsilon$, as desired.  Write $p=\{\psi_i<r_i \ : \ i=1,\ldots, k\}$.  If $p$ is not a condition, then $T_\forall\models \sup_x \prod_i (r_i\dotminus \psi_i)=0$, contradicting that $B\models T_\forall$.
\end{proof}

\begin{cor}\label{genec}
Suppose that $B$ is finite-generic.  Then $B$ is an e.c. model of $T_\forall$.
\end{cor}

\begin{proof}
Suppose that $B\subseteq C\models T_\forall$, $\varphi(x)$ is an existential formula, and $a\in B$.  By Corollary \ref{mutual}, we may find $D\models T_f$ with $C\subseteq D$.  Since $B\preceq D$, we have that $$\varphi(a)^C\leq \varphi(a)^B=\varphi(a)^D\leq \varphi(a)^C.$$  It follows that $\varphi(a)^B=\varphi(a)^C$, as desired.
\end{proof}

\begin{cor}\label{ecinfg}
Suppose that $A$ is e.c. in $B$ and $B$ is finite-generic.  Then $A$ is finite-generic.
\end{cor}

\begin{proof}
We first show that $A$ is actually elementary in $B$.  Since $A$ is e.c. in $B$, there is an embedding $B\hookrightarrow A^\u$ that restricts to the diagonal embedding $A\hookrightarrow A^\u$.  We thus have the chain
$$A\subseteq B\hookrightarrow A^\u\subseteq B^\u\hookrightarrow (A^\u)^\u\subseteq (B^\u)^\u\hookrightarrow\cdots$$ with union $A_\infty$.  Since the maps between the successive ultrapowers of $A$ are just ultrapowers of the diagonal map, we have that $A\preceq A_\infty$.  Since $B$ is finite-generic and $B^\u\models T^f$, we have that the embedding $B\hookrightarrow B^\u$ is elementary, whence so are the successive ultrapower maps.  It follows that $B\preceq A_\infty$, whence $A\preceq B$.

Now suppose that $A\subseteq C\models T^f$.  Since $A\models T_\forall$, it suffices to show that $A\preceq C$.  By Corollary \ref{genec}, $B$ is an e.c. model of $T_\forall$, whence so is $A$.  It follows that the inclusion map $A\hookrightarrow C^\u$ can be extended to a map $f:B\hookrightarrow C^\u$, which is elementary since $B$ is finite-generic.  
We then have that
$$\varphi(a)^A=\varphi(a)^B=\varphi(f(a))^{C^\u}=\varphi(a)^{C^\u}=\varphi(a)^C.$$
\end{proof}

\subsection{Finite-generic, enforceable, and prime structures}

In this subsection, we maintain the convention that forcing is with respect to the $L$-theory $T$.  The following definition contains one of the central notions of this paper.

\begin{df}
An $L$-structure $A$ is \emph{enforceable} if the property ``the reduct of the compiled structure is isomorphic to $A$'' is an enforceable property.  
\end{df}

If we want to stress the base theory $T$, we say that $A$ is enforceable with respect to $T$.  If $T$ is universal and $A$ is enforceable with respect to $T$, then by Proposition \ref{enforceuniversal}, $A$ is necessarily a model of $T$, whence may also speak of $A$ being the enforceable model of $T$.

From Proposition \ref{enforcegeneric}, we immediately have:

\begin{cor}
If $A$ is enforceable, then $A$ is finite-generic.
\end{cor}

Recall that an $L$-structure $B$ is said to be an \emph{algebraically prime model of its theory} if $B$ embeds into $C$ whenever $C\equiv B$.  $B$ is further said to be the \emph{prime model of its theory} if $B$ embeds elementarily into $C$ whenever $C\equiv B$.

The following corollaries follow immediately from Proposition \ref{chargen}.

\begin{cor}
Suppose that $B$ is a finite-generic structure and an algebraically prime model of its theory.  Then $B$ is the prime model of its theory.
\end{cor}

\begin{cor}\label{prime}
Suppose that $B$ is the enforceable structure and an algebraically prime model of its theory.  Then $B$ is the prime model of its theory.
\end{cor}

The next theorem will be the key tool in showing that certain operator algebras are the enforceable models of their universal theories.  This proof will involve a bit more model-theoretic background than the rest of this paper.
 
\begin{thm}\label{primeimpliesenforceable}
Suppose that $D$ is a finite-generic structure with respect to $\Th_\forall(D)$ and the prime model of its theory.  Then $D$ is the enforceable model of $\Th_\forall(D)$.
\end{thm}

\begin{proof}
Since being finite-generic is enforceable and any two finite-generic models with respect to $\Th_\forall(D)$ are elementarily equivalent (as $\Th_\forall(D)$ has JEP), it is enforceable that the compiled structure is a model of $\Th(D)$.

Let $S_n^i(\Th(D))$ denote the set of isolated $n$-types in $\Th(D)$, a closed subset of $S_n(\Th(D))$.  Let $c$ be an $n$-tuple of distinct constants and let $m\geq 1$ be fixed.  It is enough to show that we can enforce that, in the compiled structure, the type realized by the interpretations of $c$ is within $1/m$ of $S_n^i(\Th(D))$.  Indeed, by taking the conjunction of these countably many requirements, we can enforce that the compiled structure will be an extra canonical structure that is a model of $\Th(D)$ and that, for every $n$, a dense set of $n$-tuples realize isolated types, whence they all do; consequently, the compiled structure will be a separable, atomic model of $\Th(D)$ and hence isomorphic to $D$. 

Fix an $n$-tuple $c$ of distinct constants and $m\geq 1$.  We now describe the strategy $\exists$ can use to enforce that the type of $c$ in the compiled structure is within $1/m$ of $S^i_n(\Th(D))$.  Suppose that $\forall$ plays $p_0=\{\varphi_i(c,d)<\epsilon_i \ : \ i=1,\ldots k\}$, where $d$ is a tuple of distinct constants disjoint from $c$.  By homogeneity, we can assume that $p_0\subseteq \appdiag(D)$.  Let $[\theta(x)< \delta]$ be a logically open set contained in the ball around $\tp^D(c^D)$ of radius $1/m$.  Let $q\subseteq \Diag(D)$ be such that $q\Vdash^g \theta(c)<\delta$.  Then $p:=p_0\cup q$ is a condition extending $p_0$ and $p\Vdash^g \theta(c)<\delta$.  Thus, in the compiled structure $A$, we have that $d(\tp^A(c^A),\tp^D(c^D))< 1/m$, as desired.
\end{proof}

\begin{cor}
Suppose that $D$ is a finite-generic structure with respect to $\Th_\forall(D)$ and an algebraically prime model of its theory.  Then $D$ is the enforceable model of $\Th_\forall(D)$.
\end{cor}

\subsection{Model companions and $T^f$}

We end this section by mentioning the connection between finite-generic structures and model companions.  Recall that the theory $T'$ is a \emph{model companion} of the theory $T$ if $T_\forall=T'_\forall$ and $T'$ is model-complete, i.e. every embedding between models of $T'$ is elementary.  We note that $T$ has at most one model companion.  If $T$ is $\forall\exists$-axiomatizable, then $T$ has a model companion if and only if the class of e.c. models of $T$ is the class of models of some first-order theory, which is then necessarily the model companion of $T$.  We leave the proof of the following proposition to the reader.

\begin{prop}
The following are equivalent:
\begin{enumerate}
\item $T$ has a model companion.
\item $T^f$ is the model companion of $T$.
\item Every model of $T^f$ is finite-generic.
\end{enumerate}
\end{prop}

In particular, when $T$ has a separably categorical model companion, then the unique separable e.c. model of $T$ is necessarily enforceable.  While this phenomenon is rare in analysis, there are a few notable examples:

\begin{example}
Let $T$ be the universal theory of Banach spaces.  Then the \emph{Gurarij Banach space} $\mathbb{G}$ is the unique separable e.c. Banach space and is thus the enforceable Banach space.
\end{example}

\begin{example}
Let $T$ be the universal theory of unital abelian \cstar-algebras.  Then $C(2^\n)$ is the unique separable e.c. unital abelian \cstar-algebra and is thus the enforceable model. 
\end{example}

\section{The pseudoarc}

The original motivation for this work actually stemmed from studying the model theory of the \emph{pseudoarc} $\mathbb{P}$ and in particular trying to establish Corollary \ref{pseudoprime} below.  We recall that a \emph{continuum} is a connected compact Hausdorff space.  Note then that a compact space $X$ is a continuum if and only if $C(X)$ is projectionless.  The class of unital projectionless abelian \cstar-algebras is universally axiomatized by an $L$-theory $T$, where $L$ is the language of \cstar-algebras.  Forcing in this section is relative to the aformentioned $T$.

K.P. Hart proved the following striking fact (\cite[Lemma 2.1]{hart}) about $T$ (although not in this terminology):

\begin{fact}\label{hartfact}
If $C(X),C(Y)\models T$ are both infinite-dimensional (i.e. neither $X$ nor $Y$ are a single point), then $\Th_\forall(C(X))=\Th_\forall(C(Y))$.
\end{fact}

The \emph{pseudoarc} $\mathbb{P}$ is the unique metrizable continuum that is both \emph{hereditarily indecomposable} and \emph{chainable}.  In \cite{bankston3}, it was shown that hereditary indecomposability is an $\forall\exists$-property of models of $T$.\footnote{If $P$ is a property of continua, we will be abusive and say that $C(X)$ has property $P$ if $X$ has property $P$.}   On the other hand, the main result of \cite{EGV} shows that chainability is a $\sup\bigvee\inf$-property.  The above discussion was then used in \cite{EGV} to prove that $C(\mathbb{P})$ is an e.c.\ model of $T$, answering a question of Bankston.\footnote{This result was motivated by a result of Bankston showing that chainability is a $\forall\bigvee\exists$ property in the language of lattice bases for continua.  We should note that neither result obviously implies the other and the continuous version was needed for the aforementioned application due to the imperfect correspondence between e.c. lattice bases and co-e.c. continua.}

Proposition \ref{enforcesbvi}, Fact \ref{hartfact}, and the fact that chainability is a $\sup\bigvee\inf$-property immediately yield:

\begin{thm}
Chainability is an enforceable property.
\end{thm}

Since being e.c. is an enforceable property and e.c. models of $T$ are hereditarily indecomposable (see \cite{bankston3} again), we have the following:

\begin{cor}
$C(\P)$ is the enforceable model of $T$.
\end{cor}

The following corollary was the original motivation for this work.

\begin{cor}\label{pseudoprime}
$C(\P)$ is the prime model of its theory.
\end{cor}

\begin{proof}
By Corollary \ref{prime}, it suffices to show that $C(\P)$ is an algebraically prime model of its theory.  To see this, note that if $C(X)\equiv C(\P)$, then $X$ is hereditarily indecomposable, whence, by a result of Bellamy \cite{bellamy}, $X$ surjects onto $\P$, i.e. $C(\P)$ embeds into $C(X)$.
\end{proof}

\section{Enforceable operator algebras and embedding problems}

\subsection{II$_1$ factors}

In this subsection, $L$ denotes the language of tracial von Neumann algebras and $T$ denotes the universal $L$-theory for tracial von Neumann algebras.  (See \cite{FHS2} for details.)

\begin{thm}\label{enforceR}
$\R$ is the enforceable model of its universal theory.
\end{thm}

\begin{proof}[Proof 1]
By Theorem \ref{primeimpliesenforceable} and the well-known fact that $\R$ is the prime model of its theory \footnote{See, for example, \cite[Remark after Lemma 3.1]{GHS}.  The main point is that every embedding $\R\hookrightarrow \R^\u$ is unitarily conjungate to the diagonal embedding, and thus elementary}, it suffices to show that $\R$ is a finite-generic model of its universal theory.  Towards this end, suppose that $A$ is a finite-generic model of $\Th_\forall(\R)$.  Then $A$ is an e.c.\ model of $\Th_\forall(\R)$, hence a II$_1$ factor (see, for example, \cite{FGHS}) and thus contains $\R$.  Since $\R$ is an e.c.\ model of its universal theory (again, see \cite{FGHS}), $\R$ is finite-generic by Corollary \ref{ecinfg}.
\end{proof}

The following alternative proof is worth pointing out.  

\begin{proof}[Proof 2]
First note that hyperfiniteness is a $\sup\bigvee\inf$-property of tracial von Neumann algebras.  (This does not seem to have appeared explicitly in the literature but the proof is the same as the fact that being UHF is a $\sup\bigvee\inf$-property of \cstar-algebras; see \cite{AF}).  Thus, by Proposition \ref{enforcesbvi}, hyperfiniteness is an enforceable property for $T$.  Since being e.c. is also enforceable, we have that we can enforce that the compiled structure be a separable, hyperfinite II$_1$ factor, whence the compiled structure must be isomorphic to $\R$ by the fundamental result of Murray and von Neumann.
\end{proof}

In what follows, let $\sigma_{\operatorname{hyp}}$ denote the supremum of the countably many $\sup\bigvee\inf$-sentences that define hyperfiniteness.  Since $T$ has JEP, there is a unique value $r$ such that $\sigma_{\operatorname{hyp}}=r$ is enforceable.  We abuse notation and write $\sigma_{\operatorname{hyp}}^{T^f}$ for this unique $r$ (even though, technically, $T^f$ is a finitary theory).  We follow this abusive practice in other contexts throughout the remainder of this section.

The following was the result announced in the introduction to the paper; forcing here is with respect to $T$, the theory of tracial von Neumann algebras.


\begin{thm}\label{CEP}
The following are equivalent:
\begin{enumerate}
\item CEP has a positive solution.
\item $\sigma^{T^f}_{\operatorname{hyp}}=0$.
\item $\R$ is enforceable.
\item$\R^\u$-embeddability is enforceable.
\end{enumerate}
\end{thm}

\begin{proof}
(1)$\Rightarrow$(2):  As in Proof 2 of Theorem \ref{enforceR}, if CEP holds, then we can enforce that the compiled II$_1$ factor is hyperfinite.
(2)$\Rightarrow$(3) follows from the fact that being a II$_1$ factor is enforceable together with the aforementioned result of Murray and von Neumann, while (3)$\Rightarrow$(4) is trivial.  Finally, (4) implies (1) holds by Theorem \ref{enforcelocallyuniversal}. 
\end{proof}

\begin{rmk}\label{canonicalfactor}
As first pointed out in \cite{FHS3}, there is a locally universal II$_1$ factor.  However, locally universal II$_1$ factors are far from unique as any separable II$_1$ factor containing a locally universal II$_1$ factor is itself locally universal.  Thus, asking whether or not $\R$ is one of the many locally universal II$_1$ factors makes the connection between CEP and model theory a bit loose.  However, an enforceable II$_1$ factor, should it exist, is a \emph{canonical} object.  Thus, asking whether or not the canonical enforceable II$_1$ factor coincides with the (arguably) canonical II$_1$ factor $\R$ seems to be a more serious connection.
\end{rmk}

\subsection{Unital \cstar-algebras}

In this subsection, $L$ denotes the language for unital \cstar-algebras.
 
Recall that a \cstar-algebra $D$ is \emph{strongly self-absorbing} (or \emph{ssa} for short) if there is an isomorphism $\phi:D\to D\times D$ such that $\phi$ and $\operatorname{id}_D\otimes 1_D$ are approximately unitarily equivalent $*$-homomorphisms.  It is a well-known consequence of the definition that every embedding $D\hookrightarrow D^\u$ is unitarily conjugate to the diagonal embedding, and thus elementary.  As a result, ssa algebras are e.c.\ models of their universal theories and the prime models of their full theories.  Particularly important ssa algebras are the Cuntz algebra $\O_2$, the universal UHF algebra $\Q$, and the Jiang-Su algebra $\mathcal{Z}$.

\begin{thm}
Strongly self absoring algebras are the enforceable models of their universal theories.
\end{thm}

\begin{proof}
Suppose that $D$ is an ssa algebra.  Since $D$ is the prime model of its theory, it suffices, by Proposition \ref{primeimpliesenforceable}, to show that $D$ is a finitely-generic model of $\Th_\forall(D)$.   Let $A$ be a finitely-generic model of $\Th_\forall(D)$.  By Corollary \ref{genec}, $A$ is an e.c. model of $\Th_\forall(D)$, whence $A\otimes D\cong A$ by \cite[Lemma 2.3]{ssa}.  Thus $D$ is e.c. in $A$, whence $D$ is finitely-generic by Corollary \ref{ecinfg}.
\end{proof}

\begin{proof}[Alternate proofs for $D=\O_2$ and $\Q$]
Suppose first that $D=\O_2$.  Since nuclearity is a $\sup\bigvee\inf$-property (see \cite{munster}), we can use Proposition \ref{enforcesbvi} to show that we can enforce the compiled structure to be nuclear, whence embeddable in $\O_2$.  Since the compiled structure can also be forced to be an e.c. model of $\Th_\forall(\O_2)$, it follows that the compiled structure is e.c. in $\O_2$ and thus isomorphic to $\O_2$ by \cite[Theorem 2.14]{KEP}.

In the case that $D=\Q$, we argue in the same way, using that being UHF is a $\sup\bigvee\inf$-property (see \cite{AF}).  We can thus enforce that the compiled structure be an e.c.\ subalgebra of $\Q$, which thus forces\footnote{No pun intended.} it to be isomorphic to $\Q$.
\end{proof}

Let $T$ denote the universal $L$-theory axiomatizing the class of unital \cstar-algebras.  In the rest of this subsection, forcing is with respect to $T$.

Recall that the Kirchberg Embedding Problem (KEP) asks whether every \cstar-algebra embeds into an ultrapower of $\O_2$.  The proof of the following theorem is just like the proof of Theorem \ref{CEP}.  Here, $\sigma_{\nuc}$ is the supremum of the $\sup\bigvee\inf$-sentences defining nuclearity.
\begin{prop}\label{KEP}
The following are equivalent:
\begin{enumerate}
\item KEP has a positive solution.
\item $\sigma_{\nuc}^{T^f}=0$.
\item $\O_2$ is enforceable.
\item $\O_2^\u$-embeddability is enforceable.
\end{enumerate}
\end{prop}

There is one more equivalence we can add to the previous proposition, but first some terminology.  We say that a \cstar-algebra $A$ \emph{has a square root} if there is a \cstar-algebra $B$ such that $A\cong B\otimes B$ (minimal tensor product).  Clearly ssa algebras have square roots.  The following is a remark in \cite{KEP}; for the convenience of the reader, we repeat the statement and proof here:

\begin{lemma}\label{squareroot}
Suppose that $A$ is an e.c. \cstar-algebra that has a square root.  Then $A$ is simple and nuclear (and hence isomorphic to $\O_2$).
\end{lemma}

\begin{proof}
Suppose that $B$ is a square root of $A$.  A consequence of being existentially closed is that every automorphism of $A$ is approximately inner (see \cite{KEP}).  In particular, the \emph{flip automorphism} $a\otimes b\mapsto b\otimes a:A\to A$ is approximately inner; in other words, $B$ has \emph{approximately inner flip}.  This property passes to $A$ as well \cite{TW}; since having approximately inner half flip implies that $A$ is simple and nuclear (see \cite{TW} again), the result follows.
\end{proof}

\begin{cor}
KEP has a positive solution if and only if having a square root is an enforceable property of the compiled structure.
\end{cor}

\begin{rmk}
The previous discussion also makes sense in the II$_1$ factor category.  In that context, Connes showed that $\R$ is the only separable II$_1$ factor with ultraweak approximately inner flip.  The above arguments thus show that CEP has a positive solution if and only if having a square root is an enforceable property of the compiled II$_1$ factor.
\end{rmk}

In \cite[Section 7]{FAE}, it was shown that the \emph{local lifting property}, or LLP for short, of Kirchberg is captured by a family $(\sigma_m)$ of $L_{\omega_1,\omega}$-sentences\footnote{It is left as an open question there whether or not LLP is a $\sup\bigvee\inf$-property.}:  a \cstar-algebra $A$ has the LLP if and only if $(\sup_m \sigma_m)^A=0$.  Let $\sigma_{\LLP}:=\sup_m \sigma_m$.  We can thus ask:  what is $\sigma_{\LLP}^{T^f}$?

First suppose that $\sigma_{\LLP}^{T^f}>0$, so we can enforce that the compiled structure does not have LLP.  Since the compiled model can also be forced to be e.c., and thus has the \emph{weak expectation property}, or WEP for short (see \cite{KEP}), we get that the compiled structure can be forced to have WEP and not LLP, yielding a (potentially) new example of a \cstar-algebra with WEP but not LLP.  (See \cite{jungepisier} for the first example.)

Next suppose that $\sigma_{\LLP}^{T^f}=0$.  If $\sigma_{\nuc}^{T^f}=0$, then KEP has a positive solution.  Otherwise, $\sigma_{\nuc}^{T^f}>0$, whence we can enforce that the compiled structure is not nuclear but has both LLP and WEP, providing a positive answer to the so-called \emph{weak QWEP conjecture} (see \cite{FAE} for more on this).


\subsection{Unital stably finite \cstar-algebras}  Once again, $L$ denotes the language for unital \cstar-algebras.  Except for the last results in this subsection, $T$ now denotes the universal $L$-theory axiomatizing the class of unital, \emph{stably finite} \cstar-algebras.

Recall that the MF problem asks whether or not every stable finite \cstar-algebra embeds into an ultrapower of the universal UHF algebra $\Q$.  In what follows, $\sigma_{\operatorname{UHF}}$ is the supremum of the $\sup\bigvee\inf$-sentences defining being UHF and $\sigma_{\operatorname{QD}}$ is the supremum of the $\sup\bigvee\inf$-sentences defining being quasidiagonal (see \cite{munster}).

\begin{thm}
The following are equivalent:
\begin{enumerate}
\item The MF problem has a positive solution.
\item $\sigma_{\operatorname{UHF}}^{T^f}=0$.
\item $\Q$ is enforceable.
\item $\sigma_{\operatorname{QD}}^{T^f}=0$.
\item $\Q^\u$-embeddability is enforceable.
\end{enumerate}
\end{thm}

\begin{rmk}
As pointed out in \cite{Robinson}, it is currently unknown as to whether or not the class of unital, stably finite \cstar-algebras has JEP.  Thus, in the previous proposition, it is unknown as to whether or not $\sigma_{\operatorname{UHF}}^{T^f}$ even exists!  Similarly, while in the cases of CEP and KEP, we could have proven that (5) implies (1) using that being e.c. is enforceable and using JEP, we can not use such an argument in the case of the MF problem, and thus, at the moment, the use of Theorem \ref{enforcelocallyuniversal} really is needed.
\end{rmk}

The \emph{quasidiagonality problem} (or QD problem) asks whether or not every stably finite \emph{nuclear} algebra is quasidiagonal (equivalently, by the Choi-Effros Lifting Theorem, $\Q^\u$-embeddable).  The best progress towards resolving the QD problem is the main result of \cite{TWW}, which states that a unital, simple, stably finite, nuclear algebra satisfying the \emph{Universal Coefficient Theorem} (UCT) is $\Q^\u$-embeddable.  Since being simple and nuclear are both $\sup\bigvee\inf$-properties (see \cite{munster}), if we assume that every nuclear \cstar-algebra has the UCT, then we can add $$\sigma_{\operatorname{\nuc}}^{T^f}=\sigma_{\operatorname{simple}}^{T^f}=0$$ to the above list of equivalent formulations of the MF problem.

As pointed out in \cite{Robinson}, the stably finite version of Lemma \ref{squareroot} holds:  if $A$ is a stably finite \cstar-algebra that is e.c.\ for the class of stably finite algebras and $A$ has a square root, then $A$ is simple and nuclear (and is furthermore isomorphic to $\Q$ if $A$ is UCT).  Consequently, we have:

\begin{cor}
Assume that every nuclear \cstar-algebra is UCT.  Then the MF problem has a positive solution if and only if having a square root is an enforceable property of the compiled structure.
\end{cor}

The previous discussion makes one wonder about the logical status of the UCT.  In particular, the following question comes to mind:

\begin{question}
Is the UCT an $L_{\omega_1,\omega}$ property of nuclear \cstar-algebras?
\end{question}

The next theorem spells out the precise difference between the QD problem and the MF problem:

\begin{thm}
The following are equivalent:
\begin{enumerate}
\item The MF problem has a positive solution.
\item The conjunction of the following two statements:
\begin{enumerate}
\item The QD problem has a positive solution.
\item Nuclearity is an enforceable property.
\end{enumerate}
\end{enumerate}
\end{thm}

\begin{proof}
(1) implies (2)  since the MF problem having a positive solution implies that $\Q$ is enforceable.  For (2) implies (1), note that once we know that nuclearity is enforceable, then a positive solution to the QD problem implies that quasidiagonality is enforceable.
\end{proof}

\begin{rmk}
In \cite{Robinson}, it is conjectured that the only possible stably finite algebra that is both nuclear and e.c. for the class of stably finite algebras is $\Q$.  If this conjecture holds, then we see that the MF problem having a positive solution is simply equivalent to nuclearity being enforceable.
\end{rmk}

A problem related to the MF problem is whether or not every stably finite \cstar-algebra has a trace.  Of course, if the MF problem has a positive solution, then the aforementioned problem has a positive solution.  There is a connection with enforceability:

\begin{thm}
Every stably finite \cstar-algebra has a trace if and only if having a trace is an enforceable property of the compiled structure.
\end{thm}

\begin{proof}
Suppose that we can enforce that the compiled structure has a trace.  Let $A$ be a stably finite \cstar-algebra.  It suffices to show, given any condition $p\subseteq \appdiag(A)$, that $p$ can be satisfied in a tracial stably finite algebra.  Indeed, by writing $\appdiag(A)$ as an increasing union of conditions, we can then satisfy $\appdiag(A)$ in an ultraproduct of tracial stably finite algebras, which is itself tracial.  It follows that $A$ can be embedded in a tracial algebra and is thus, itself, tracial.

Now given a condition $p$ from $\appdiag(A)$, view $p$ as $\forall$'s first move in the game and have $\exists$ follow its strategy to ensure that the compiled structure is tracial.  We then have that $p$ is realized in a tracial algebra, as desired.
\end{proof}

A related question is whether or not every \emph{quasitrace} on a stably finite \cstar-algebra is necessarily a trace.  It is known that every stably finite \cstar-algebra has a quasitrace, so a positive answer to the previous question implies that every stably finite \cstar-algebra has a trace.  

In \cite[Proposition 31]{Robinson}, it is shown that, in the language $L_\tau$ obtained by adding a unary function symbol $\tau$ to the above language $L$, the class of structures $(A,\tau)$, where $A$ is a \cstar-algebra and $\tau$ is a quasitrace on $A$, is universally axiomatizable, say by the universal $L_\tau$-theory $T_\tau$.  Moreover, it is easy to see that the class of such pairs where $\tau$ is actually a trace is also universally axiomatizable.  Arguing in the same way as in the preceding theorem, we see that:

\begin{thm}
Let forcing be with respect to $T_\tau$.  Then every quasitrace on a stably finite \cstar-algebra is a trace if and only if it is enforceable that the quasitrace on the compiled structure is a trace. 
\end{thm} 

Haagerup \cite{haag} showed that quasitraces on exact \cstar-algebras are traces, so if one can enforce (with respect to $T_\tau$) that the compiled structure is exact, then every quasitrace on a stably finite \cstar-algebra is a trace.

We end this subsection by mentioning the case of stably finite, \emph{projectionless} algebras.
\begin{thm}
Let $T_{\operatorname{sfp}}$ be the universal $L$-theory for unital, projectionless, stably finite \cstar-algebras and let forcing be with respect to $T_{\operatorname{sfp}}$.  Then the following are equivalent:
\begin{enumerate}
\item Every unital, projectionless, stably finite algebra is $\mathcal{Z}^\u$-embeddable.
\item $\mathcal{Z}$ is enforceable.
\item $\mathcal{Z}^\u$-embeddability is enforceable.
\end{enumerate}
\end{thm}

%

%
%

\subsection{Operator spaces and systems}

In this section, we let $L$ denote the language of operator spaces and $T$ the universal $L$-theory for operator spaces.  (See \cite[Appendix B]{KEP}.)  Let $\NG$ denote the so-called \emph{noncommutative Gurarij space}, which is the Fraisse limit of the finite-dimensional $1$-exact operator spaces.  (See \cite{lupini} for other equivalent descriptions of $\NG$.)  It is readily checked that the proof that nuclearity is a $\sup\bigvee\inf$-property of \cstar-algebras also establishes the same fact for operator spaces.  Since every operator space embeds into an ultrapower of $\NG$, it follows that we can enforce that the compiled operator space be nuclear.

In \cite[Section 5.6]{lupini}, building on ideas from \cite{GL}, it was shown that $\NG$ is the unique e.c.\ operator space that is also $1$-exact (in particular nuclear).  Since we can also enforce that the compiled operator space be e.c., we have:

\begin{prop}
$\NG$ is the enforceable model of $T$.
\end{prop}

If we instead work in the operator system category, the analog of $\NG$ is the \emph{Gurarij operator system} $\GS$, whose model-theoretic properties were laid out in \cite{GL}.  The exact same arguments show that $\GS$ is the enforceable model of the theory of operator systems.
  

\section{The dichotomy theorem}

\subsection{The dichotomy theorem and embedding problems revisited}

The goal of this chapter is to prove the following theorem, which is the continuous analog of \cite[Theorem 4.2.6]{hodges}:

\begin{thm}
Suppose that $T$ is an $\forall\exists$-axiomatizable theory JEP.  Then exactly one of the following happens:
\begin{enumerate}
\item For every enforceable property $P$, there are continuum many nonisomorphic models of $T$ with property $P$.
\item $T$ has an enforceable model.
\end{enumerate}
\end{thm}


The remaining subsections will be devoted to the proof of the dichotomy theorem.  However, before we turn to the proof, let us mention how this theorem suggests a new strategy for providing a positive solution to the embedding problems from the previous section.  Let us first consider CEP.

\

\noindent \textbf{Step 1:}  Find an enforceable property $P$ of II$_1$ factors shared by fewer than continuum many nonisomorphic separable II$_1$ factors.

\

By the Dichotomy Theorem and Step 1, there is an enforceable II$_1$ factor $\mathcal E$.  

\

\noindent \textbf{Step 2:}  Show that the enforceable II$_1$ factor $\mathcal E$ must be isomorphic to $\R$.

\

Clearly one (or both!) of these steps must be difficult, but it is not clear to us which step that is.  That being said, as mentioned in Remark \ref{canonicalfactor}, since being an enforceable II$_1$ factor is such a canonical property, it is hard to envision one existing without it being isomorphic to arguably the most canonical II$_1$ factor $\R$.

\begin{rmk}
In trying to establish Step 1, one should \emph{not} try to show that there is a first-order property $P$ that has fewer than continuum many nonisomorphic separable models.  Indeed, as shown in \cite{FHS3}, given any II$_1$ factor $M$, there are continuum many nonisomorphic separable II$_1$ factors elementarily equivalent to $M$.
\end{rmk}

The above strategy can be stated in an analogous fashion for the KEP.  In connection with Step 2 for the KEP, the following remark seems in order.
  
\begin{rmk}
Suppose that $\mathcal{E}$ is the enforceable \cstar-algebra.  Then $\mathcal{E}$ is finite-generic, whence every embedding $\mathcal{E}\hookrightarrow \mathcal{E}^\u$ is elementary.  Thus, assuming the Continuum Hypothesis (CH), any two embeddings of $\mathcal{E}$ into $\mathcal{E}^\u$ are conjugate by an automorphism of $\mathcal{E}^\u$.  \emph{If} one can show that these automorphism are implemented by unitaries \emph{and} that $\mathcal{E}\equiv \mathcal{E}\otimes \mathcal{E}$, then, by \cite[Theorem 2.14]{ssa}, it follows that $\mathcal{E}$ is ssa and hence $\mathcal{E}\cong \O_2$.  Since the question of whether or not $\mathcal{E}$ and $\O_2$ are isomorphic is absolute (see \cite{farah}), the assumption of CH is harmless here.  
\end{rmk}

The case of the MF problem is different in that, as mentioned in the last section, it is currently unknown whether or not the class of stably finite \cstar-algebras has JEP.  \emph{If we assume} that the class of stably finite \cstar-algebras has JEP and the above strategy then worked, we would conclude that the MF problem has a positive solution.  Of course, if the MF problem has a positive solution, then every stably finite \cstar-algebras has a trace, which itself implies that the minimal tensor product of two stably finite \cstar-algebras is stably finite \cite[Theorem 2.4]{haag}, so the above strategy in the stably finite case would amount to a strategy for solving the following (possibly outlandish):

\begin{conj}
The following are equivalent:
\begin{enumerate}
\item The class of stably finite \cstar-algebas has JEP.
\item The MF problem has a positive solution.
\item Every stably finite \cstar-algebra has a trace.
\end{enumerate}
\end{conj}


\subsection{The topometric space $S_n^\exists(T)$}

In this subsection, we let $T$ be an $\forall\exists$-axiomatizable $L$-theory with JEP.  For $A\models T$ and a tuple $a$ from $A$, set 
$$\etp^A(a)=\{\varphi(x)=0 \ : \ \varphi \text{ existential and }\varphi^A(a)=0\}.$$  We call $\etp^A(a)$ the \emph{existential type of $a$ in $A$}.  For $n\geq 1$, an \emph{existential $n$-type} is the existential type of an $n$-tuple from a model of $T$.

The following lemma will prove useful a number of times.

\begin{lemma}\label{JEP}
Suppose that $A,A'\models T$ and $a$ and $a'$ are tuple from $A$ and $A'$ respectively of the same length such that $\etp^{A'}(a')\subseteq \etp^A(a)$.  Then there is $A''\models T$ and embeddings $i:A\to A''$ and $j:A'\to A''$ such that $i(a)=j(a')$.
\end{lemma}

\begin{proof}
Let $C$ and $D$ denote two disjoint countably infinite sets of new constant symbols and expand $A$ and $A'$ to canonical $L(C)$- and $L(D)$-structures $A^+$ and $(A')^+$ respectively.  Without loss of generality, $a$ and $a'$ are named by tuples of constants, say $c$ and $d$.  It is enough to show that 
$$T\cup \Diag(A)\cup \Diag(A')\cup \{d(c,d)=0\}$$ is satisfiable.  Fix $\varphi(c,c_1)=0$ from $\Diag(A)$ and $\psi(d,d_1)=0$ from $\Diag(A')$, with $c$ and $d$ disjoint tuples of constants and likewise for $c'$ and $d'$.  Also fix $\epsilon>0$.  By compactness, it is enough to show that $$T\cup\{\varphi(c,c_1)=0,\psi(d,d_1)\leq \epsilon,d(c,d)=0\}$$ is satisfiable.  Since $\inf_y\psi(x,y)=0$ belongs to $\etp^{A'}(a')\subseteq \etp^A(a)$, there is $e\in A$ such that $\psi(a,e)\leq \epsilon$.  Expand $A$ to an $L(C\cup D)$-structure $A^{++}$ by further expanding $A^+$ to interpret $d$ as $a$ and $d_1$ as $e$ and the other constants by anything.  It follows that $A^{++}$ satisfies the last displayed set of conditions.
\end{proof}

\begin{df}
An existential type is \emph{maximal} if it is not properly contained in any other existential type.  For $n\geq 1$, we set $$S_n^\exists(T):=\{\etp^A(a) \ : \ \etp^A(a) \text{ is a maximal $n$-type}\}.$$
\end{df}

We will use letters like $\pi$ and $\rho$ to denote elements of $S_n^\exists(T)$.

\begin{lemma}
Elements of $S_n^\exists(T)$ are precisely the existential $n$-types $\etp^A(a)$ where $A\models T$ is e.c.
\end{lemma}

\begin{proof}
First suppose that $\pi\in S_n^\exists(T)$.  Write $\pi=\etp^A(a)$ for some $A\models T$ and $a\in A$.  Let $B\supseteq A$ be an e.c. model of $T$.  Then $\pi\subseteq \etp^B(a)$; by maximality, $\pi=\etp^B(a)$.

Conversely, suppose that $\pi=\etp^A(a)$ for $A\models T$ e.c.  Suppose that $\pi\subseteq \etp^{A'}(a')$ for some $A'\models T$.  By Lemma \ref{JEP}, there is $A''\models T$ and $i:A\to A''$, $j:A'\to A''$ such that $i(a)=j(a')$.  Now suppose that $\varphi(x)=0$ belongs to $\etp^{A'}(a')$.  Then $\varphi(x)=0$ belongs to $\etp^{A''}(j(a'))=\etp^{A''}(i(a))$.  Since $A$ is e.c., it follows that $\varphi(x)=0$ belongs to $\pi$.
\end{proof}

\begin{df}
Given an existential formula $\varphi(x)$, with $x$ an $n$-tuple of variables, and $\epsilon>0$, let $[\varphi<\epsilon]$ denote the set of elements $\pi\in S_n^\exists(T)$ such that, writing $\pi=\etp^A(a)$ for $A\models T$ e.c., then $\varphi^A(a)<\epsilon$.
The \emph{logic topology} on $S_n^\exists(T)$ has, as basic open neighborhoods of $\pi(x)$, sets of the form $[\varphi<\epsilon]$, where $\varphi=0$ belongs to $\pi(x)$ and $\epsilon>0$.
\end{df}

\begin{lemma}
The logic topology on $S_n^\exists(T)$ is Hausdorff.
\end{lemma}

\begin{proof}
Suppose that $\pi,\rho\in S_n^\exists(T)$ are distinct.  Without loss of generality, we may take an existential formula $\varphi$ such that $\varphi=0$ belongs to $\pi$ but not to $\rho$.  By maximality of $\rho$, there must be some $\epsilon>0$ such that $\rho\cup\{\varphi\leq \epsilon\}$ is not satisfiable, whence, by compactness, there is some $\psi(x)$ and some $\delta>0$ such that $\psi(x)=0$ belongs to $\rho$ and $\{\psi(x)\leq \delta,\varphi(x)\leq \epsilon\}$ is not satisfiable.  It follows that $[\varphi<\epsilon]$ and $[\psi<\delta]$ are disjoint open neighborhoods of $\pi$ and $\rho$ respectively.
\end{proof}

There is also a natural metric on $S_n^\exists(T)$.

\begin{df}
For $\pi,\rho\in S_n^\exists(T)$, set
$$d(\pi,\rho):=\inf\{d(a,b) \ : \ a,b\in A\models T, \pi=\etp^A(a), \rho=\etp^A(b)\}.$$
\end{df}

Note that JEP is needed to ensure that $\pi$ and $\rho$ are realized in a common model, whence $d(\pi,\rho)<\infty$.  Note also that, by saturation, the infimum in the above definition is actually a minimum.

\begin{lemma}
$d$ is a metric on $S_n^\exists(T)$.
\end{lemma}

\begin{proof}
Reflexivity and symmetry are clear.  For transitivity, fix $\epsilon>0$ and take $A,A'\models T$, $a,b\in A$, $b',c\in A'$ such that $\pi=\etp^A(a)$, $\rho=\etp^A(b)=\etp^{A'}(b')$, and $\sigma=\etp^{A'}(c)$ with $d(a,b)\leq d(\pi,\rho)+\epsilon$ and $d(b',c)\leq d(\rho,\sigma)+\epsilon$.  By Lemma \ref{JEP}, there is $A''\models T$ and embeddings $i:A\to A''$ and $j:A'\to A''$ such that $i(b)=j(b')$.  By maximality, $\pi=\etp^{A''}(i(a))$, $\rho=\etp^{A''}(i(b))=\etp^{A''}(j(b'))$ and $\sigma=\etp^{A''}(j(c))$.  It follows that
$$d(\pi,\sigma)\leq d(i(a),j(c))\leq d(i(a),i(b))+d(j(b'),j(c))\leq d(\pi,\rho)+d(\rho,\sigma)+2\epsilon.$$  Since $\epsilon>0$ was arbitrary, we are done.
\end{proof}

Recall from \cite{topo} that a \emph{topometric space} is a triple $(X,\tau,d)$, where $(X,\tau)$ is a Hausdorff topological space, $(X,d)$ is a metric space, and the following two conditions holds:
\begin{itemize}
\item The metric topology refines the topology $\tau$.
\item $d$ is $\tau$-lower semi-continuous, i.e., for all $r>0$, the set $$\{(x,y)\in X^2 \ : \ d(x,y)\leq r\}$$ is $(\tau\times \tau)$-closed.
\end{itemize}

\begin{prop}
$S_n^\exists(T)$ is a topometric space.
\end{prop}

\begin{proof}
It is clear that $d$ refines the logic topology.  For the second item, suppose that $d(\pi,\rho)>r$.  Then $T\cup \pi(x)\cup \rho(y)\cup \{d(x,y)\leq r\}$ is not satisfiable, so by compactness, there are existential formulae $\varphi$ and $\psi$ and $\delta>0$ such that $\varphi=0$ belongs to $\pi$, $\psi=0$ belongs to $\rho$, and $T\cup\{\varphi(x)<\delta,\psi(y)<\delta,d(x,y)\leq r\}$ is not satisfiable.  It follows that $d(\pi',\rho')>r$ for any $\pi'\in[\varphi<\delta]$ and any $\rho'\in [\psi<\delta]$.    
\end{proof}

\subsection{Isolated existential types and $e$-atomic models}

We continue to assume that $T$ is an $\forall\exists$-axiomatizble theory with JEP.  As discussed in \cite{topo}, in topometric spaces there are two appropriate notions of isolated point.  For a topometric space $(X,\tau,d)$, $x\in X$ is called:
\begin{itemize}
\item \emph{$d$-isolated} if the two topologies agree at $x$;
\item \emph{weakly $d$-isolated} if, for every $\epsilon>0$, the open ball $B(x,\epsilon)$ centered at $x$ of radius $\epsilon$ has nonempty $\tau$-interior. 
\end{itemize}

Clearly every $d$-isolated point is weakly $d$-isolated.  In general topometric spaces, these notions may be distinct.  However, we have:

\begin{lemma}
In $S_n^\exists(T)$, every weakly $d$-isolated point is $d$-isolated.
\end{lemma}

\begin{proof}
The corresponding fact for $S_n(T)$ is \cite[Proposition 12.5]{bbhu}; we note that the proof applies to $S_n^\exists(T)$ verbatim.
\end{proof}

We may thus just refer to \emph{isolated types} in $S_n^\exists(T)$.
\begin{cor}
The set of isolated types in $S_n^\exists(T)$ is metrically closed.
\end{cor}

\begin{proof}
In \cite[Lemma 2.2]{topo}, it is shown that the set of weakly $d$-isolated points in an arbitrary topometric space is metrically closed.
\end{proof}

Suppose that $\pi$ is isolated, $\epsilon>0$, and $O$ is a logically open set contained in $B(\pi,\epsilon)$.  If $B\models T$ is e.c. and $\etp^B(b)\in O$, then a priori, all we are guaranteed is that there are realizations of $\pi$ and $\etp^B(b)$ in some (possibly different) e.c. model of $T$ that are within $\epsilon$ of each other.  Our next goal is to show that this can in fact be improved by showing that, after possibly shrinking $O$, if $\etp^B(b)\in O$, then there is some realization of $\pi$ \emph{in $B$} that is within $\epsilon$ of $b$.  First, a preliminary lemma.

\begin{lemma}
Fix $\pi\in S_n^\exists(T)$ and $\epsilon>0$.  Suppose that $O$ is a logically open neighborhood of $\pi$ contained in $B(\pi,\epsilon)$.  Suppose that $B\models T$ is e.c. and $b\in B$ is such that $\etp^B(b)\in O$.  Then for all logically open $U$ containing $\pi$, there is $b'\in B$ such that $\etp^B(b')\in U$ and $d(b,b')<\epsilon$.
\end{lemma}

\begin{proof}
Fix a logically open neighborhood $U$ of $\pi$.  By hypothesis, there is e.c. $C\models T$ and $c,d\in C$ such that $\pi=\etp^C(c)$, $\etp^B(b)=\etp^C(d)$, and $d(c,d)<\epsilon$.  By Lemma \ref{JEP}, there is e.c. $D\models T$ and $i:B\to D$ and $j:C\to D$ such that $i(b)=j(d)$.  Thus, $\pi=\etp^D(j(c))$, $\etp^B(b)=\etp^D(i(b))$, and $d(i(b),j(c))<\epsilon$.  The result now follows from the fact that $B$ is e.c. in $D$.
\end{proof}


\begin{prop}\label{defset}
Suppose that $\pi\in S_n^\exists(T)$ is isolated.  Then for all $\epsilon>0$, there is a logically open set $O$ such that if $B\models T$ is e.c., $b\in B$ and $\etp^B(b)\in O$, then there is $c\in B$ with $\pi=\etp^B(c)$ and $d(b,c)<\epsilon$.
\end{prop}

\begin{proof}
Take $K\in \n$ such that $\sum_{k=K}^\infty 2^{-k}<\epsilon$.  For $k\geq K$, let $O_k$ be a logically open neighborhood of $\pi$ contained in $B(\pi,2^{-k})$.  Set $O:=O_K$.  We claim that $O$ is as desired.  Suppose that $b\in B$ and $\etp^B(b)\in O$.  By the previous lemma, there is $b_1\in B$ such that $\etp^B(b_1)\in O_{K+1}$ and $d(b,b_1)<2^{-K}$.  By the previous lemma again, there is $b_2\in B$ such that $\etp^B(b_2)\in O_{K+2}$ and $d(b_1,b_2)<2^{-K+1}$.  Continuing in this way, it follows that $(b_k)_{k\geq K}$ is a Cauchy sequence in $B$.  Set $c=\lim b_k$.  We have that $d(b,b')\leq \sum_{k\geq K}d(b_k,b_{k+1})<\epsilon$ and $\etp^B(c)=\pi$.
\end{proof}

\begin{df}
$A\models T$ is called \emph{$e$-atomic} if, for every $n\geq 1$ and every $n$-tuple $a$ from $A$,  $\etp^A(a)$ is an isolated element of $S_n^\exists(T)$.
\end{df}

Note that, in particular, every existential type realized in an $e$-atomic model is maximal, so $e$-atomic models are e.c. 


%


The proof of the following fact follows the outline of the corresponding fact for atomic models of complete theories given by Bradd Hart in his online lecture notes \cite[Lecture 7]{bhart}.  We recall our outstanding assumption that $T$ has JEP.

\begin{prop}\label{uniqueatomic}
If $A,B\models T$ are both separable and $e$-atomic, then $A\cong B$.
\end{prop}

\begin{proof}
We will produce sequences
$$a_0^0, a_0^1a_1^1, a_0^2a_1^2a_2^2,\ldots,$$
and
$$b_0^0, b_0^1b_1^1, b_0^2b_1^2b_2^2,\ldots,$$
from $A$ and $B$ respectively such that:
\begin{enumerate}
\item for all $n\geq k$, $\etp^A(a_0^n\cdots a_k^n)=\etp^B(b_0^n\cdots b_k^n)$;
\item for all $k\leq n$, $d(a_k^n,a_k^{n+1}),d(b_k^n,b_k^{n+1})\leq 2^{-n}$; consequently, for every $k$, $(a_k^n)_{n\geq k}$ and $(b_k^n)_{n\geq k}$ are Cauchy sequences in $A$ and $B$ respectively whose limits we shall denote by $a_k$ and $b_k$;
\item $(a_k)$ and $(b_k)$ are dense in $A$ and $B$ respectively.
\end{enumerate}
Assuming that these sequences have been produced, then the map $a_k\mapsto b_k$ clearly extends to an isomorphism from $A$ to $B$.

Let $(c_k)$ and $(d_k)$ enumerate countable, dense subsets of $A$ and $B$ respectively.  We perform the usual back-and-forth style argument, at each stage putting either some $c_k$ in the sequences of $a$'s or some $d_k$ in the sequence of $b$'s, revisiting each $c_k$ and $d_k$ infinitely often.  We start by setting $a_0^0:=c_0$.  Let $O$ be a logically open set contained in $B(\etp^A(c_0),\frac{1}{2})$.  By JEP, there is $b\in B$ such that $\etp^B(b)\in O$.  By Proposition \ref{defset}, there is $b'\in B$ such that $\etp^B(b')=\etp^A(c_0)$.  We set $b_0^0$ to be this $b'$.

Now suppose that we have constructed $a_0^na_1^n\cdots a_n^n$ and $b_0^nb_1^n\cdots b_n^n$ and we are considering $c_k$.  We set $a_0^na_1^n\cdots a_{n+1}^{n+1}:=a_0^na_1^n\cdots a_n^nc_k$, $\pi:=\etp^A(a_0^na_1^n\cdots a_n)$, and $\rho=\etp^A(a_0^na_1^n\cdots a_nc_k)$.  Let $O$ be a basic logically open set as guaranteed to exist by Proposition \ref{defset} for $\pi$ and $2^{-n}$, say $O=[\varphi(x_0,\dots,x_n,y)<\epsilon]$.  Since $A\models \inf_y \varphi(a_0^n,\ldots,a_n^n,y)<\epsilon$, by the inductive assumption, we have that $B\models \inf_y \varphi(b_0^n,\ldots,b_n^n,y)<\epsilon$.  By Proposition \ref{defset}, there is $b_0^{n+1}b_1^{n+1}\cdots b_{n+1}^{n+1}\in B$ such that $\rho=\etp^B(b_0^{n+1}\cdots b_{n+1}^{n+1})$ and $d(b_i^n,b_i^{n+1})\leq 2^{-n}$ for $i\leq n$.

We clearly have (1) and (2).  It remains to show (3).  Fix $\epsilon>0$ and take $N\in \n$ such that $\sum_{n\geq N}2^{-n}<\epsilon$.  Suppose $c_k$ is visited at stage $n>N$.  Then $d(a_n,c_k)<\epsilon$; since the $(c_k)$'s are dense, we get that $(a_n)$ is dense.  The same argument holds for $(b_n)$.

\end{proof}

A ``forth-only'' version of the above proof shows:

\begin{prop}\label{atomicembeds}
If $A$ is an $e$-atomic model of $T$, then $A$ embeds into all e.c. models of $T$.
\end{prop}

We will see later (Corollary \ref{embedsimpliesatomic}) that the converse of this proposition holds.  Now that we have settled the uniqueness of separable $e$-atomic models, the question of existence remains.  We first note a necessary condition.

\begin{lemma}\label{atomicdense}
If $T$ has an $e$-atomic model, then the isolated types in $S_n^\exists(T)$ are logically dense for all $n\geq 1$.
\end{lemma}

\begin{proof}
Let $A$ be an $e$-atomic model of $T$.  Fix a non-empty logically open set $[\varphi<\epsilon]$.  Then there is an e.c.\ model $B\models T$ such that $(\inf_x\varphi(x))^B<\epsilon$.  By JEP, $(\inf_x \varphi(x))^A<\epsilon$.  If $a\in A$ is such that $\varphi^A(a)<\epsilon$, then the isolated type $\etp^A(a)$ belongs to $[\varphi<\epsilon]$.
\end{proof}

What is more important is that the converse holds.  In fact:

\begin{lemma}\label{enforceatomic}
Suppose that the isolated types in $S_n^\exists(T)$ are logically dense for all $n\geq 1$.  Then the property that the compiled structure is $e$-atomic is enforceable.
\end{lemma}

\begin{proof}
By the conjunction lemma and the fact that isolated elements of $S_n^\exists(T)$ are metrically closed, it is enough to show that, for any $\delta>0$ and any tuple $c$ of distinct witnesses, that we can enforce that $\etp^A(c)$ is maximal and within $\delta$ of an isolated type.  Suppose that $\forall$ opens with $p_0=\{\psi_i(c,d)<\epsilon_i \ : \ i=1,\ldots,k\}$, where $c$ and $d$ are disjoint tuples of constants.  Fix $\epsilon>0$ and $\delta_i$ such that $\delta_i+\epsilon<\epsilon_i$.  By assumption, there is an isolated maximal existential type $\pi(x)$ contained in $[\inf_y \max_i(\psi_i(x,y)\dotminus \delta_i)<\epsilon]$.  Suppose that $[\theta(x)<\eta]$ is a neighborhood of $\pi$ contained in $B(\pi,\delta)$.  Since $T\cup p_0\cup \{\theta(x)<\eta\}$ is consistent, by Lemma \ref{existentialforcing},  $\exists$ can play $p_1\supseteq p_0$ such that $p_1\Vdash^g \theta(c)<\eta$.  It follows that in the compiled structure we can force that $A$ is e.c. and $d(\etp^A(c),\pi)<\delta$.  
\end{proof}

Combining Proposition \ref{uniqueatomic} with Lemmas \ref{atomicdense} and \ref{enforceatomic}, we obtain:

\begin{cor}\label{denseenforceable}
Suppose that the isolated points in $S_n^\exists(T)$ are dense for all $n\geq 1$.  Then $T$ has an enforceable model.  
\end{cor}

\begin{cor}
Suppose that $T$ has an $e$-atomic model $A$.  Then $A$ is the enforceable model of $T$.
\end{cor}

\subsection{Games with many boards}

Once again, we assume that $T$ is an $\forall\exists$-axiomatizable theory with JEP.

In the proof of the dichotomy theorem, it is important to extend our game to the setting where we have ``many boards.''  More concretely, let us first consider the game with two boards, which is played exactly as before, except each player plays two conditions $p_{i+1}^1,p_{i+1}^2$ extending the previous players conditions $p_i^1,p_i^2$.  It is important to note that the two boards are independent of one another.  At the end, providing both players played definitive sequences, the players will have compiled two structures, say $A_1^+$ and $A_2^+$ with reducts $A_1$ and $A_2$.  Given a property $R$ of pairs of structures, we say that $R$ is enforceable if $\exists$ has a winning strategy that ensures that the pair of compiled structures has property $R$.

The following lemma is obvious but worth pointing out.

\begin{lemma}

\

\begin{enumerate}
\item If $P$ and $Q$ are enforceable properties of structures, then it is enforceable that the compiled pair $(A_1^+,A_2^+)$ is such that $A_1^+$ has $P$ and $A_2^+$ has $Q$.
\item If $(R_i)$ is a family of countably many enforceable properties of pairs of structures, then the conjunction of the $R_i$'s is also enforceable.
\end{enumerate}
\end{lemma} 

For us, the main proposition about the two-board game is the following:

\begin{prop}\label{simultaneous}
It is enforceable that the only maximal existential types realized in both $A_1$ and $A_2$ are $e$-isolated.
\end{prop}

\begin{proof}
Fix $\delta>0$ and tuples of distinct constants $c$ and $d$.  By the conjunction lemma, it suffices to show that if $\etp^{A_1}(c)=\etp^{A_2}(d)$, then the $\delta$ ball around this common existential type $\pi$ contains a logically open set.  

Suppose that $\forall$ starts by playing $p_0^1$ and $p_0^2$.  Write $$p_0^1=\{\psi_{i1}(c,b_1)<\epsilon_{i1} \ : \ i=1,\ldots,k\}$$ and $$p_0^2=\{\psi_{i2}(c,b_2)<\epsilon_{i2} \ : \ i=1,\ldots,k\}.$$

First suppose that $O:=\bigcap_{i=1}^k [\inf_{y}\psi_{i1}(x,y)<\epsilon_{i1}]\subseteq B(\rho;\frac{\delta}{2})$ for some $\rho\in O$.  It follows that if the game ends with $\etp^{A_1}(c)=\etp^{A_2}(d)=\pi$, then $\pi\in O$, whence $d(\pi,\rho)<\frac{\delta}{2}$, and $O\subseteq B(\pi,\delta)$.

If the first case does not apply, then there are $\rho_1,\rho_2\in O$ such that $d(\rho_1,\rho_2)\geq \frac{\delta}{2}$.  Take $\sigma\in \bigcap_{i=1}^k [\inf_{y}\psi_{i2}(x,y)<\epsilon_{i2}]$.  Without loss of generality, we then have that $d(\rho_1,\sigma)\geq \frac{\delta}{4}$.  Since $S_n^\exists(T)$ is a topometric space, there are logically open sets $[\chi<\eta]$ and $[\theta<\zeta]$ containing $\rho_1$ and $\sigma$ respectively such that $d([\chi<\eta],[\theta<\zeta])\geq \frac{\delta}{8}$.  By Lemma \ref{existentialforcing}, $\exists$ may respond by playing $p_1^1\supseteq p_0^1$ and $p_1^2\supseteq p_0^2$ such that $p_1^1\Vdash^g \chi( c)<\eta$ and $p_1^2\Vdash^g \theta( d)<\eta$.  It is clear then that in the compiled structures, $d(\etp^{A_1}(c),\etp^{A_2}(d))\geq \frac{\delta}{8}$.      
\end{proof}

\begin{cor}\label{embedsimpliesatomic}
Suppose that $T$ is $\forall\exists$-axiomatizable and $A$ is an e.c. model of $T$ that embeds into all e.c. models of $T$.  Then $A$ is $e$-atomic and hence enforceable.
\end{cor}


\begin{rmk}
The previous corollary gives an alternative proof of the fact that ssa algebras are enforceable models of their universal theories.
\end{rmk}

The other game that we will need is the following ``splitting game.''  While we will not present the most general version of the game, this is the only version that we will need in the proof of the dichotomy theorem.  In this game, $\forall$ starts by playing a condition $p_0^\emptyset$ and $\exists$ responds by playing $p_1^\emptyset\supseteq p_0^\emptyset$.  Now, $\forall$ responds with two extensions $p_2^{0},p_2^{1}\supseteq p_1^\emptyset$ and $\exists$ responds with single extensions $p_3^0\supseteq p_2^0$ and $p_3^1\supseteq p_2^1$.  More generally, for every $s\in 2^{<\omega}$, assume $\forall$ has played conditions $p_i^s$.  $\exists$ then responds with $p_{i+1}^s\supseteq p_i^s$ and then $\forall$ responds with two extensions $p_{i+2}^{s0},p_{i+2}^{s1}\supseteq p_{i+1}^s$.

At the end of a play, we have a tree of plays, where nodes at even levels have precisely one extension while nodes at odd levels have precisely two extensions.  Provided each infinite path through the tree is a definitive play of the original one-board game, we have a family $(A_\alpha^+ \ : \ \alpha\in 2^\omega)$ of continuum many compiled structures.  Given a property $R$ of families of structures indexed by $2^\omega$, we hope it is clear to the reader how to make sense of the statement that $R$ is an enforceable property.  

The main fact that we will need about the splitting game is the following.  Its proof is not difficult (just a notational mess) and is exactly the same as its classical counterpart (see \cite[Theorem 4.1.5]{hodges}) so we omit the proof.

\begin{prop}\label{splitting}
Let $R$ be an enforceable property of pairs of structures.  Let $P$ be the property of families $(B_\alpha^+ \ : \ \alpha\in 2^\omega)$ of structures that states that $(B_\alpha^+,B_\beta^+)$ has property $R$ whenever $\alpha\not=\beta$.  Then $P$ is an enforceable property.
\end{prop}

\subsection{Proof of the dichotomy theorem}

We now have all the ingredients needed to prove the dichotomy theorem.  If the $e$-isolated types are dense for all $n\geq 1$, then we know that we have an enforceable model by Corollary \ref{denseenforceable}.  So assume now that the $e$-isolated types are not dense and fix an enforceable property $P$.  Take a basic logically open set $[\varphi<\epsilon]$ that contains no $e$-isolated type.  

We play the splitting game from the previous section.  Let $p_0^\emptyset$ be a condition such that $p_0^{\emptyset}\Vdash^g \varphi(c)<\epsilon$ and then have $\forall$ play future stages any way they want.  We obtain models $(A_\alpha^+ \ : \ \alpha\in 2^\omega)$ with $a_\alpha:=c^{A_\alpha^+}$.  By Propositions \ref{simultaneous} and \ref{splitting}, $\exists$ can enforce that each $A_\alpha^+$ is an e.c. model of $T$ with property $P$ such that $\varphi(a_\alpha)^{A_\alpha}<\epsilon$ and that the only types realized in distinct $A_\alpha$'s are $e$-isolated.  It remains to show that $A_\alpha$ and $A_\beta$ are not isomorphic for $\alpha\not=\beta$.  Let $\pi_\alpha:=\tp^{A_\alpha}(a_\alpha)$.  If $\pi_\alpha$ is realized in $A_\beta$, then $\pi_\alpha$ is $e$-isolated, contradicting the fact that $\pi_\alpha\in [\varphi<\epsilon]$.  

\begin{rmk}
From the dichotomy theorem and Lemma \ref{enforceatomic}, we see that enforceable models are $e$-atomic.  In particular, $C(\P)$ is an $e$-atomic model of the theory of unital, projectionless, abelian \cstar-algebras.  By Proposition \ref{atomicembeds}, it follows that $C(\P)$ embeds into $C(X)$ whenever $C(X)$ is e.c.  This is a special case of the result of Bellamy used in the proof that $C(\P)$ is prime, namely that any hereditarily indecomposable continuum surjects onto $\P$.  It would be interesting to see if some further elaborations of the ideas used in this paper could be used to give a model-theoretic proof of Bellamy's result.
\end{rmk}
%

\end{document}